\documentclass[a4paper]{article}
\usepackage{amsmath}
\usepackage{amssymb}
\usepackage{graphicx}
\usepackage{amsthm}
\usepackage{bbm}
\usepackage{enumerate}
\usepackage{accents}

\theoremstyle{definition}
\newtheorem{mydef}{Definition}[section]

\theoremstyle{plain}
\newtheorem{myprop}[mydef]{Proposition}
\newtheorem{mythm}[mydef]{Theorem}
\newtheorem{mylemma}[mydef]{Lemma}
\newtheorem{mycor}[mydef]{Corollary}

\theoremstyle{remark}
\newtheorem{myremark}[mydef]{Remark}

\numberwithin{equation}{section}

\newcommand{\Prob}[0]{\mathbb{P}}
\newcommand{\E}[0]{\mathbb{E}}
\newcommand{\Ind}[0]{\mathbbm{1}}
\newcommand{\States}[0]{\mathcal{X}}
\newcommand{\Reals}[0]{\mathbb{R}}
\newcommand{\F}[0]{\mathcal{F}}
\newcommand{\J}[0]{\mathbb{J}}
\newcommand{\Q}[0]{\mathbb{Q}}
\newcommand{\ubar}[1]{\underaccent{\bar}{#1}}

\begin{document}

\title{Ergodic Backward Stochastic Difference Equations}
\author{Andrew L. Allan and Samuel N. Cohen\footnote{Research supported by the Oxford--Man Institute for Quantitative Finance.}\\
\smallskip Mathematical Institute, University of Oxford\\
andrew.allan@keble.ox.ac.uk\\
samuel.cohen@maths.ox.ac.uk}
\date{\today}
\maketitle

\begin{abstract}
We consider ergodic backward stochastic differential equations in a discrete time setting, where noise is generated by a finite state Markov chain. We show existence and uniqueness of solutions, along with a comparison theorem. To obtain this result, we use a Nummelin splitting argument to obtain ergodicity estimates for a discrete time Markov chain which hold uniformly under suitable perturbations of its transition matrix. We conclude with an application of this theory to a treatment of an ergodic control problem.


\end{abstract}

\section{Introducing Discrete Time BSDEs}

\subsection{Introduction}

Over the last 25 years, Backward Stochastic Differential Equations (BSDEs) have been extensively researched, and established as a fundamental object in mathematical finance and stochastic control. Here, one is typically interested in the solution process $(Y,Z)$ of an equation of the form
$$Y_t - \int_t^T f(\omega,u,Y_u,Z_u) du + \int_t^T Z_udW_u = \xi,$$
for some generator (or driver) function $f$ and some $\F_T$-measurable terminal condition $\xi$, where $W$ is an $n$-dimensional Brownian motion.

When the cost of a control problem is considered in an ergodic manner, that is, it is the long run behaviour which is important, the associated equations turn out to be examples of a variation of these equations, named `Ergodic BSDEs', which typically take the form
$$Y_T = Y_t - \int_t^T \big(f(X_u,Z_u) - \lambda\big) du + \int_t^T Z_udW_u.$$
Having been first introduced in Fuhrman, Hu and Tessitore \cite{FuhrmanHuTessitore2009}, the theory of these equations and their link to ergodic control problems have been researched over the last few years; see Richou \cite{Richou2009}, Debussche, Hu and Tessitore \cite{DebusscheHuTessitore2011} and Cohen and Fedyashov \cite{CohenFedyashov}. In Cohen and Hu \cite{CohenHu2013}, these equations were considered in the case where noise is generated by a continuous time, countable state Markov chain.

This paper considers the discrete time analogue of these equations, where noise is generated by a discrete time, finite state Markov chain. In Section 2 we shall extend the existing theory of discrete time BSDEs, as given in Cohen and Elliott \cite{CohenElliott2010}, to a class of infinite horizon `discounted' BSDEs. Section 3 is a brief digression to the study of uniformly ergodic Markov chains in discrete time, in which we shall show how to obtain ergodicity estimates for a Markov chain which are robust under a suitable perturbation of its transition matrix. These estimates have not been previously obtained in a discrete time setting and are of independent interest in the study of the ergodic properties of discrete time Markov chains.

We shall then make use of these estimates in Section 4 to construct a suitable limiting procedure with which to prove the existence of bounded Markovian solutions to discrete time Ergodic BSDEs. We will also prove a comparison theorem, and observe the relationship between the Markovian solution and the ergodic measure of the underlying Markov chain. Finally, in Section 5 we shall show how our theory can be applied to a particular ergodic control problem.

\subsection{Finite horizon BSDEs}

We will begin by introducing the theory of finite horizon, discrete time BSDEs, as established in \cite{CohenElliott2010}.

Consider an underlying discrete time, finite state stochastic process $X$. Without loss of generality, we may assume that $X$ takes values in the standard basis of $\Reals^N$, for some $N \in \mathbb{N}$. That is, for each $t \in \{0,1,2,\ldots\}$,
$$X_t \in \States := \{e_1,e_2,\ldots,e_N\},$$
where $e_k = (0,0,\ldots,0,1,0,\ldots,0)^{\ast} \in \Reals^N$, and $[\, \cdot \,]^{\ast}$ denotes vector (or matrix) transposition.

Let $(\Omega,\mathcal{F},\{\mathcal{F}_t\}_{t \geq 0},\Prob)$ be a filtered probability space, where $\mathcal{F}_t$ is the completion of the $\sigma$-algebra generated by the process $X$ up to time $t$, and assume that $\mathcal{F} = \mathcal{F}_{\infty} := \bigvee_{t \geq 0} \mathcal{F}_t$.

\begin{myremark}
We have not assumed that $\F$ is complete, nor that $\F_0$ contains all the $\Prob$-null sets in $\F_{\infty}$. Later, we will consider our processes under a variety of different probability measures, and we do not wish for $\Prob$-null events in the tail $\sigma$-algebra to be known at finite times, as these events may occur with positive probability under a new measure.
\end{myremark}

Note that the atoms of $\F_t$ are sets of positive measure of the form
$$\{\omega' \in \Omega : X_s(\omega') = X_s(\omega) \ \, \text{for all} \ \, s \leq t\}$$
for some $\omega \in \Omega$, up to a null set in $\F_t$. We will find that much of our analysis can be done by restricting our attention to individual atoms.

For $t \geq 1$, let $M_t := X_t - \E[X_t \, | \, \mathcal{F}_{t-1}]$, which we note defines a martingale difference sequence $M$.

We now define a backward stochastic difference equation (BSDE) as an equation of the form
\begin{equation}\label{eq:BSDE}
Y_t - \sum_{t \leq u < T} f(\omega,u,Y_u,Z_u) + \sum_{t \leq u < T} Z_u^{\ast}M_{u+1} = \xi,
\end{equation}
where $T>0$ is a finite deterministic terminal time, $f : \Omega \times \{0,\ldots,T-1\} \times \Reals \times \Reals^N \rightarrow \Reals$ is an adapted map, and $\xi : \Omega \rightarrow \Reals$ is an $\mathcal{F}_T$-measurable random variable. By ``$f$ is adapted'' we mean that, for all $t \in \{0,\ldots,T-1\}$, the map from $\Omega \times \Reals \times \Reals^N \rightarrow \Reals$ given by $(\omega,y,z) \mapsto f(\omega,t,y,z)$ is $\mathcal{F}_t \otimes \mathcal{B}(\Reals) \otimes \mathcal{B}(\Reals^N)$-measurable. Since there are only finitely many possible paths for $X$ up to time $t$, $\mathcal{F}_t$ is sufficiently coarse that all $\F_t$-measurable real-valued functions belong to $L^1(\Omega,\mathcal{F}_t,\Prob)$ and $L^{\infty}(\Omega,\mathcal{F}_t,\Prob)$.

A solution to the BSDE (\ref{eq:BSDE}) is a pair $(Y,Z)$ of adapted processes, taking values in $\Reals \times \Reals^N$, such that
$$Y_t(\omega) - \sum_{t \leq u < T} f(\omega,u,Y_u(\omega),Z_u(\omega)) + \sum_{t \leq u < T} Z_u^{\ast}(\omega) M_{u+1}(\omega) = \xi(\omega)$$
holds for all $t \in \{0,\ldots,T\}$ and almost all $\omega$.

\begin{mydef}\label{DefnSeminorms}
We will denote by $\| \cdot \|_{M_{t+1}}$ the stochastic seminorm on the space of $\Reals^N$-valued, $\F_t$-measurable random variables given by
$$\| Z_t \|_{M_{t+1}}^2 := Z_t^{\ast} \, \E \big[ M_{t+1}M_{t+1}^{\ast} \, \big| \, \F_t \big] Z_t = \E \big[ (Z_t^{\ast}M_{t+1})^2 \, \big| \, \F_t \big],$$
and write $Z_t \sim_{M_{t+1}} Z'_t$ whenever $\| Z_t - Z'_t \|_{M_{t+1}} = 0$ a.s. We also denote by $\| \cdot \|_M$ the seminorm on the space of adapted processes in $\Reals^N$ given by
$$\| Z \|_M^2 := \E \Bigg[ \sum_{0 \leq u < T} \| Z_u \|_{M_{u+1}}^2 \Bigg] = \sum_{0 \leq u < T} \E \big[ (Z_u^{\ast}M_{u+1})^2 \big],$$
and write $Z \sim_M Z'$ whenever $\| Z - Z' \|_M = 0$.
\end{mydef}

It is clear that $Z \sim_M Z'$ if and only if $Z_t \sim_{M_{t+1}} Z'_t$ for all $t \in \{0,\ldots,T-1\}$. It is also easy to see that $\sim_M$ and $\sim_{M_{t+1}}$ are both equivalence relations.

We will construct the solutions of our BSDE by making use of the following Martingale Representation Theorem from Elliott and Yang \cite{ElliottYang1994}.

\begin{mythm}\label{martrepthm}
Let $L$ be a real-valued $\{\F_t\}_{t \geq 0}$-adapted martingale. Then there exists an adapted $\Reals^N$-valued process $Z$ such that
$$L_t = L_0 + \sum_{0 \leq u < t} Z_u^{\ast}M_{u+1}$$
for all $t$. Further, the process $Z$ is unique up to equivalence $\sim_M$.
\end{mythm}

\begin{proof}
By the Doob--Dynkin Lemma (see for example Shiryaev \cite[p.174]{Shiryaev1996}), for each $t$, there exists a function $g_t : \States \times \States \times \cdots \times \States \rightarrow \Reals$ such that
$$L_{t+1} = L_t + g_t(X_0,X_1,\ldots,X_{t+1}).$$
Define $Z_t$ as the $\F_t$-measurable, $\Reals^N$-valued random variable with components given by $e_k^{\ast}Z_t = g_t(X_0,X_1,\ldots,X_t,e_k)$ for each $e_k \in \States$, and notice that
$$Z_t^{\ast}X_{t+1} = g_t(X_0,X_1,\ldots,X_t,X_{t+1}) = L_{t+1} - L_t.$$
As $L$ is a martingale, we see that $Z_t^{\ast}M_{t+1} = Z_t^{\ast}X_{t+1}$. Hence, $L_{t+1} = L_t + Z_t^{\ast}M_{t+1}$, and the result follows. The uniqueness of $Z$ follows easily from the definition of $\sim_M$.
\end{proof}

\begin{mycor}\label{CorMartRepThm}
Let $W$ be a real-valued $\F_{t+1}$-measurable random variable such that $\E [W \, | \, \F_t] = 0$. Then there exists an $\Reals^N$-valued $\F_t$-measurable random variable $Z_t$ such that
$$W = Z_t^{\ast}M_{t+1} \ \ \text{a.s.}$$
Further, this variable is unique up to equivalence $\sim_{M_{t+1}}$.
\end{mycor}

\begin{proof}
Simply consider the martingale $L$ defined by $L_s = \Ind_{\{s>t\}}W$, and apply Theorem~\ref{martrepthm}.
\end{proof}

The following theorem provides sufficient conditions for the existence of a unique solution to our finite horizon BSDE. In general, these conditions are also necessary for a unique solution to exist for all terminal conditions $\xi$, as shown in \cite[Corollary 2]{CohenElliott2010}.

\begin{mythm}\label{FiniteHorizonExistUniq}
Suppose our driver $f$ is such that the following two assumptions hold:
\begin{enumerate}[(i)]
\item For any $Y$, if $Z \sim_M Z'$, then $f(\omega,t,Y_t,Z_t) = f(\omega,t,Y_t,Z'_t)$ a.s. for all $t$.
\item For any $z \in \Reals^N$, any $t$ and for almost all $\omega$, the map
$$y \mapsto y - f(\omega,t,y,z)$$
is a bijection $\Reals \rightarrow \Reals$.
\end{enumerate}
Then, for any $\F_T$-measurable terminal condition $\xi$, the BSDE (\ref{eq:BSDE}) has an adapted solution $(Y,Z)$. Moreover, this solution is unique, up to indistinguishability for $Y$ and $\sim_M$ for $Z$.
\end{mythm}

\begin{proof}
We shall construct the solution using backward induction. Note that $Y_T = \xi$ is the unique solution for $Y$ at time $T$. Suppose that we have found a solution for $Y$ at time $t+1$. The one step dynamics of (\ref{eq:BSDE}) are given by
\begin{equation}\label{eq:FiniteHorizonOneStep}
Y_t - f(\omega,t,Y_t,Z_t) + Z_t^{\ast}M_{t+1} = Y_{t+1}.
\end{equation}
Taking an $\E[\, \cdot \, | \, \F_t]$ expectation gives
\begin{equation}\label{eq:OneStepCondExp}
Y_t - f(\omega,t,Y_t,Z_t) = \E[Y_{t+1} \, | \, \F_t],
\end{equation}
and substituting this back into (\ref{eq:FiniteHorizonOneStep}), we obtain
\begin{equation}\label{eq:MartDifference}
Z_t^{\ast}M_{t+1} = Y_{t+1} - \E[Y_{t+1} \, | \, \F_t].
\end{equation}
It follows that any adapted solution $(Y_t,Z_t)$ at time $t$ must satisfy both (\ref{eq:OneStepCondExp}) and (\ref{eq:MartDifference}). By Corollary~\ref{CorMartRepThm}, there exists an $\Reals^N$-valued $\F_t$-measurable random variable $Z_t$ such that (\ref{eq:MartDifference}) holds, and this variable is unique up to equivalence $\sim_{M_{t+1}}$.

With this $Z_t$, it follows from assumption (\textit{i}) that equation (\ref{eq:OneStepCondExp}) is now uniquely determined as an equation in $Y_t$. By assumption (\textit{ii}), for almost all $\omega$, this equation has a unique solution $Y_t(\omega)$. The pair $(Y_t,Z_t)$ is then the unique solution of the BSDE at time $t$, and we can therefore construct the full solution $(Y,Z)$ by backward induction. Since for each $t$, $Y_t$ is unique up to equality a.s. and $Z_t$ is unique up to $\sim_{M_{t+1}}$, it follows that the solution $(Y,Z)$ is unique up to indistinguishability for $Y$, and $\sim_M$ for $Z$.
\end{proof}

\subsection{BSDEs on Markov chains}

Let us now consider the case when our underlying process $X$ is a discrete time, finite state Markov chain under $\Prob$. To keep our notation somewhat consistent with the continuous time theory in \cite{CohenHu2013}, we shall define the transition matrix $A_t = [a_{ij}]$ of the Markov chain $X$ by $a_{ij} = \Prob (X_{t+1}=i \mid X_t=j)$, so that the columns of the matrix give the probability distributions associated with the transitions of the chain. Note that $X$ is allowed to be time-inhomogeneous, so that its transition matrix $A_t$ can vary (deterministically) through time.

Notice that now $\E[X_{t+1} \, | \, \F_t] = A_tX_t$, so that $M_{t+1} = X_{t+1} - A_tX_t$.

\begin{mydef}
We shall say that a driver $f$ is \textit{Markovian} if at every time $t$, the value of $f$ only depends on $\omega$ through the value of $X_t$. That is, $f$ is Markovian if it can be written as $f(\omega,t,y,z) = \tilde{f}(X_t,t,y,z)$ for some function $\tilde{f} : \States \times \mathbb{N} \times \Reals \times \Reals^N \rightarrow \Reals$.
\end{mydef}

\begin{mylemma}\label{FiniteHorizonMarkovSoln}
Assume that $X$ is a Markov chain with transition matrix $A_t$, and let $f$ be a Markovian driver which satisfies the conditions of Theorem~\ref{FiniteHorizonExistUniq}. Suppose that $\xi = \varphi(X_T)$ for some deterministic function $\varphi : \States \rightarrow \Reals$, and let $(Y,Z)$ be the solution of the corresponding BSDE (\ref{eq:BSDE}). Then there exists a function $v : \{0,1,\ldots,T\} \times \States \rightarrow \Reals$ such that, for each $t$, $Y_t = v(t,X_t)$ and the $k$\textsuperscript{th} component of $Z_t$ is given by $e_k^{\ast}Z_t = v(t+1,e_k)$ (cf. \cite[Theorem 3.2]{CohenSzpruch2012}).

In particular, the solution $(Y,Z)$ is also Markovian in the sense that $Y_t$ and $Z_t$ are both deterministic functions of $X_t$ (and in fact $Z_t$ is completely deterministic). Moreover, writing $\mathbf{v}_{t+1}$ for the vector in $\Reals^N$ with entries $e_k^{\ast}\mathbf{v}_{t+1} = v(t+1,e_k)$, the function $v$ satisfies
\begin{equation}\label{eq:veqn1}
v(t,e_k) - \tilde{f}(e_k,t,v(t,e_k),\mathbf{v}_{t+1}) - \mathbf{v}_{t+1}^{\ast}A_te_k = 0
\end{equation}
for all $t \in \{0,\ldots,T-1\}$ and all $e_k \in \States$ such that $\Prob(X_t = e_k) > 0$.
\end{mylemma}

\begin{proof}
We shall prove the existence of the function $v$ by backward induction in $t$. Let $v(T,\, \cdot \,) := \varphi(\, \cdot \,)$, so that $Y_T = v(T,X_T)$.

Suppose that for some $t$ we have found a function $v(t+1,\, \cdot \,) : \States \rightarrow \Reals$ such that $Y_{t+1} = v(t+1,X_{t+1})$. Recall from the proofs of Theorem~\ref{martrepthm} and Theorem~\ref{FiniteHorizonExistUniq}, that the $k$\textsuperscript{th} component of $Z_t$ is given by $e_k^{\ast}Z_t = g_t(X_0,X_1,\ldots,X_t,e_k)$, where, in this case,
\begin{equation}\label{eq:ZdoesnotdependonX}
g_t(X_0,X_1,\ldots,X_t,X_{t+1}) = v(t+1,X_{t+1}) - \sum_{i=1}^{N} v(t+1,e_i)e_i^{\ast}A_tX_t.
\end{equation}
Note that the sum on the right-hand side of (\ref{eq:ZdoesnotdependonX}) does not depend on $X_{t+1}$, so that this term corresponds to the addition of a multiple of $\mathbf{1}$ to $Z_t$, where $\mathbf{1}$ denotes the vector in $\Reals^N$ with all entries equal to $1$. However, we infer from Definition~\ref{DefnSeminorms} that $\|\mathbf{1}\|_{M_{t+1}} = 0$, so the addition of multiples of $\mathbf{1}$ does not change the value of $Z_t$ up to equivalence $\sim_{M_{t+1}}$. We may therefore ignore this term, and deduce that $e_k^{\ast}Z_t = v(t+1,e_k)$ for all $e_k \in \States$.

By assumption, the driver $f$ can be written as $f(\omega,t,y,z) = \tilde{f}(X_t,t,y,z)$ for some function $\tilde{f}$. From the proof of Theorem~\ref{FiniteHorizonExistUniq}, $Y_t$ is given by the (a.s. unique) solution of the equation:
$$Y_t - \tilde{f}(X_t,t,Y_t,Z_t) = \E[Y_{t+1} \, | \, \F_t].$$
Hence, since $Z_t$ and $\E[Y_{t+1} \, | \, \F_t]$ are both deterministic functions of $X_t$, it follows that the same is true of $Y_t$, so there must exist a function $v(t,\, \cdot \,) : \States \rightarrow \Reals$ such that $Y_t = v(t,X_t)$.
We therefore deduce the existence of the required function $v$ by backward induction. The equation (\ref{eq:veqn1}) follows from the one-step dynamics of (\ref{eq:BSDE}).
\end{proof}

\section{Infinite Horizon BSDEs}

\subsection{Discounted BSDEs}

We now introduce a `discounted' BSDE as an infinite horizon equation of the form:
\begin{equation}\label{eq:DiscountedBSDE}
Y_T = Y_t - \sum_{t \leq u < T} \big( f(\omega,u,Z_u) - \alpha Y_u \big) + \sum_{t \leq u < T} Z_u^{\ast}M_{u+1},
\end{equation}
where $\alpha > 0$ is a fixed constant. Here $T$ is no longer a fixed terminal time. Instead, both $t$ and $T$ may take any finite values such that $0 \leq t < T$. A solution of this equation is an adapted process $(Y,Z)$ such that (\ref{eq:DiscountedBSDE}) holds a.s. for every $0 \leq t < T$. In this section we will show that this equation admits unique bounded solutions for a suitable class of drivers. Similar results are considered in a continuous time setting in Royer \cite{Royer2004} and in \cite{CohenHu2013}, however the discrete time case has not been previously considered.

We will return to the case when $X$ is a Markov chain in Section 2.2. However, the theory we present here for discounted BSDEs holds much more generally.

\begin{myremark}
Recall Definition~\ref{DefnSeminorms}. In light of our infinite horizon setting, we will henceforth write $Z \sim_M Z'$ if $Z_t \sim_{M_{t+1}} Z'_t$ for all $t \geq 0$.
\end{myremark}

\begin{myprop}\label{ChangeMeasure}
Suppose that for any $t \geq 0$, $y \in \Reals$, $z,z' \in \Reals^N$ and any atom $F \in \F_t$, the driver $f$ satisfies
\begin{equation}\label{eq:Comparison}
f(\omega,t,y,z) - f(\omega,t,y,z') > \min_{i \in \J_t^F} \big\{(z - z')^{\ast}(e_i - \E[X_{t+1} \, | \, \F_t])\big\}
\end{equation}
on $F$, where $\J_t^F := \{i : \Prob(X_{t+1} = e_i \, | \, F) > 0\}$, unless
$$\min_{i \in \J_t^F} \big\{(z - z')^{\ast}(e_i - \E[X_{t+1} \, | \, \F_t])\big\} = 0,$$
in which case the inequality can be replaced by an equality.
Then, for any adapted processes $Y,Z,Z'$, there exists a probability measure $\Q$ on $(\Omega,\F)$, which is equivalent to $\Prob$ on $\F_t$ for every finite $t$, and such that
$$\tilde{M}_t := -\sum_{0 \leq u < t} \big( f(\omega,u,Y_u,Z_u) - f(\omega,u,Y_u,Z'_u) \big) + \sum_{0 \leq u < t} (Z_u - Z'_u)^{\ast}M_{u+1}$$
is a martingale under $\Q$.
\end{myprop}

\begin{proof}
We will define $\Q$ inductively in $t$. First let $\Q|_{\F_0} = \Prob|_{\F_0}$.

Next, suppose that we have defined $\Q$ on $\F_t$ for some $t \geq 0$. Let $F$ be an atom of $\F_t$ (so that $\Q(F) > 0$), and let
$$S_t^F := \big\{(Z_t - Z'_t)^{\ast}(e_i - \E_{\Prob}[X_{t+1} \, | \, \F_t]) : i \in \J_t^F\big\},$$
where here all variables are evaluated on $F$. Note that
\begin{equation}\label{eq:MartDiffExpZero}
\E_{\Prob}\big[(Z_t - Z'_t)^{\ast}(X_{t+1} - \E_{\Prob}[X_{t+1} \, | \, \F_t]) \, \big| \, F\big] = 0,
\end{equation}
and hence that $\min \big(S_t^F\big) \leq 0$ and $\max \big(S_t^F\big) \geq 0$. Suppose that $\min \big(S_t^F\big) < 0$. By assumption, we can swap the roles of $z$ and $z'$ in (\ref{eq:Comparison}), and deduce that
$$f(\omega,t,Y_t,Z_t) - f(\omega,t,Y_t,Z'_t) \in \text{conv}^{\circ}\big(S_t^F\big),$$
where here conv$^{\circ}(\cdot)$ denotes the interior of the convex hull of a subset of $\Reals$.

If $\min \big(S_t^F\big) = 0$, then it follows from (\ref{eq:MartDiffExpZero}) that $S_t^F = \{0\}$, and hence, by assumption, that
$$f(\omega,t,Y_t,Z_t) - f(\omega,t,Y_t,Z'_t) = 0.$$
It follows from the definition of convexity that, in either case, for each $i \in \J_t^F$, there exists $\mu_i^F > 0$ such that $\sum_{i \in \J_t^F} \mu_i^F = 1$, and
\begin{equation}\label{eq:ReassignProbMass}
f(\omega,t,Y_t,Z_t) - f(\omega,t,Y_t,Z'_t) = \sum_{i \in \J_t^F} \mu_i^F (Z_t - Z'_t)^{\ast}(e_i - \E_{\Prob}[X_{t+1} \, | \, \F_t]).
\end{equation}
Define $\Q(\{X_{t+1} = e_i\} \cap F) := \mu_i^F\Q(F)$ for $i \in \J_t^F$, and $\Q(\{X_{t+1} = e_i\} \cap F) := 0$ for $i \notin \J_t^F$.

The measure $\Q$ is now defined on the whole of $\F_{t+1}$ and therefore, by induction, on $\bigcup_{s \geq 0} \F_s$. By Carath\'{e}odory's extension theorem (see for example \cite[p.152]{Shiryaev1996}), there exists a unique extension of $\Q$ to a probability measure on $\F = \sigma \big(\bigcup_{s \geq 0} \F_s\big)$. By construction, $\Q$ is equivalent to $\Prob$ on $\F_t$, for any $t \geq 0$. Further, for any $t$, $F$ and $i \in \J_t^F$ as above, $\Q(\{X_{t+1} = e_i\} \, | \, F) = \mu_i^F$.

Note that
\begin{align*}
\E_{\Q}\big[\tilde{M}_{t+1} - \tilde{M}_t \, \big| \, \F_t\big] &= -\big(f(\omega,t,Y_t,Z_t) - f(\omega,t,Y_t,Z'_t)\big) \\
&\ \ \ \ \ \ + (Z_t - Z'_t)^{\ast}\big(\E_{\Q}[X_{t+1} \, | \, \F_t] - \E_{\Prob}[X_{t+1} \, | \, \F_t]\big).
\end{align*}
It follows from (\ref{eq:ReassignProbMass}) that the above expression is equal to zero on every atom $F$ of $\F_t$. Hence, as required, $\tilde{M}$ is a martingale under $\Q$.
\end{proof}

\begin{myremark}
Note that the measure $\Q$ constructed in the proof of Proposition~\ref{ChangeMeasure} is by no means uniquely defined, as in general there will be many ways of choosing the constants $\mu_i^F$ such that the required conditions hold. We also note that in general $\Q$ will not be equivalent to $\Prob$ on $\F$, which is why it was important that we did not assume that $\F_0$ contains all the $\Prob$-null sets in $\F$.
\end{myremark}

\begin{mydef}\label{DefnLipschitz}
We shall say that a driver $f$ is Lipschitz in $z$, if there exists $L > 0$ such that, for any $t$ and any $\F_t$-measurable random variables $Y_t,Z_t,Z'_t$,
$$\big| f(\omega,t,Y_t,Z_t) - f(\omega,t,Y_t,Z'_t) \big| \leq L \, \big\| Z_t - Z'_t \big\|_{M_{t+1}} \ \ \text{a.s.}$$
\end{mydef}

We will now proceed to prove a result on the existence and uniqueness of bounded solutions to our discounted BSDE. The proof is an adaptation of that of \cite[Theorem 2.10]{CohenHu2013} into our discrete time setting.

\begin{mythm}\label{SolnDiscountedBSDE}
Let $\alpha > 0$, and suppose that the driver $f : \Omega \times \mathbb{N} \times \Reals^N \rightarrow \Reals$ is independent of $y$, and satisfies the following conditions:
\begin{itemize}
\item $f$ is Lipschitz in $z$, with Lipschitz constant $L$,
\item $f$ satisfies the conditions of Proposition~\ref{ChangeMeasure}, and
\item $f(\omega,t,0)$ is uniformly bounded by some $C > 0$.
\end{itemize}
Then there exists an adapted solution $(Y,Z)$ to the (infinite horizon) discounted BSDE
\begin{equation}\label{eq:DiscountedBSDE2}
Y_T = Y_t - \sum_{t \leq u < T} \big( f(\omega,u,Z_u) - \alpha Y_u \big) + \sum_{t \leq u < T} Z_u^{\ast}M_{u+1},
\end{equation}
satisfying $|Y_t| \leq C/\alpha$ for all $t$, and this solution is unique among bounded adapted solutions.

Furthermore, if $(Y^T,Z^T)$ denotes the (unique) adapted solution to the finite horizon BSDE
\begin{equation}\label{eq:DiscountedUpToTimeT}
0 = Y_t^T - \sum_{t \leq u < T} \big( f(\omega,u,Z_u^T) - \alpha Y_u^T \big) + \sum_{t \leq u < T} \big(Z_u^T\big)^{\ast}M_{u+1},
\end{equation}
then $Y_t^T \rightarrow Y_t$ in $L^{\infty}$ as $T \rightarrow \infty$, uniformly on finite sets in $t$.
\end{mythm}

\begin{proof}
First we will show that if a bounded solution exists, then it is unique. Suppose we have two bounded solutions $(Y,Z)$ and $(Y',Z')$ to (\ref{eq:DiscountedBSDE2}). Let $\delta Y = Y - Y'$ and $\delta Z = Z - Z'$, and consider the one-step dynamics, given by
\begin{equation}\label{eq:OneStepDelta}
\delta Y_{t+1} = (1 + \alpha)\delta Y_t - \big(f(\omega,t,Z_t) - f(\omega,t,Z'_t)\big) + \delta Z_t^{\ast}M_{t+1}.
\end{equation}
By Proposition~\ref{ChangeMeasure}, there exists a probability measure $\Q_1$ such that
$$-\sum_{0 \leq u < t} \big( f(\omega,u,Z_u) - f(\omega,u,Z'_u) \big) + \sum_{0 \leq u < t} \delta Z_u^{\ast}M_{u+1}$$
is a martingale under $\Q_1$. It follows that $\delta Y_t = (1 + \alpha)^{-1}\E_{\Q_1}[\delta Y_{t+1} \, | \, \F_t]$, and hence, by induction, that
$$|\delta Y_t| \leq (1 + \alpha)^{-(r-t)}\E_{\Q_1}\big[|\delta Y_r| \, \big| \, \F_t\big]$$
for all $r > t$. Note that the above holds $\Q_1$-a.s., but since $\Q_1$ is equivalent to $\Prob$ on $\F_t$ for every finite $t$, there is no distinction here between $\Q_1$-a.s. and $\Prob$-a.s. Since $Y$ and $Y'$ are both uniformly bounded, the same is true of $\delta Y$. We then deduce, upon taking the limit as $r \rightarrow \infty$ in the above, that $\delta Y_t = 0$ a.s. It follows that $Y = Y'$, up to $\Prob$ and $\Q_1$-indistinguishability. Substituting back into (\ref{eq:OneStepDelta}), we deduce that $\delta Z_t^{\ast}M_{t+1} = 0$, and hence that $Z \sim_M Z'$.

We now proceed to prove existence. It follows from the Lipschitz condition that the driver of (\ref{eq:DiscountedUpToTimeT}) satisfies the assumptions of Theorem~\ref{FiniteHorizonExistUniq}, so the solution $(Y^T,Z^T)$ exists uniquely. First we show that $Y^T$ is bounded.

By Proposition~\ref{ChangeMeasure}, there exists a probability measure $\Q_2^T$ such that
$$-\sum_{0 \leq u < t} \big( f(\omega,u,Z_u^T) - f(\omega,u,0) \big) + \sum_{0 \leq u < t} \big(Z_u^T\big)^{\ast}M_{u+1}$$
is a martingale under $\Q_2^T$, where we let $Z^T_t = 0$ for $t \geq T$. It follows from the one-step dynamics of (\ref{eq:DiscountedUpToTimeT}) that
$$Y^T_t = (1 + \alpha)^{-1}\E_{\Q_2^T}\big[Y^T_{t+1} + f(\omega,t,0) \, \big| \, \F_t\big].$$
Using the above as the basis of an inductive argument yields
$$Y^T_t = \E_{\Q_2^T}\Bigg[(1 + \alpha)^{-(r-t)}Y^T_r + \sum_{k=0}^{r-t-1} (1 + \alpha)^{-(k+1)}f(\omega,t+k,0) \ \Bigg| \ \F_t\Bigg]$$
for all $t < r \leq T$.
Setting $r = T$ and using the fact that $Y^T_T = 0$, we deduce that $Y^T$ is uniformly bounded by $C/\alpha$.

Let $T' > T$. Taking the difference of the one-step dynamics of $Y^T$ and $Y^{T'}$ we obtain
\begin{align}
Y^T_{t+1} - Y^{T'}_{t+1} &= (1 + \alpha)\big(Y^T_t - Y^{T'}_t\big) - \Big(f\big(\omega,t,Z^T_t\big) - f\big(\omega,t,Z^{T'}_t\big)\Big) \label{eq:OneStepTT'} \\
&\ \ \ + \big(Z^T_t - Z^{T'}_t\big)^{\ast}M_{t+1}. \nonumber
\end{align}
By Proposition~\ref{ChangeMeasure}, there exists a probability measure $\Q_3^{T,T'}$ such that
$$-\sum_{0 \leq u < t} \Big(f\big(\omega,u,Z^T_u\big) - f\big(\omega,u,Z^{T'}_u\big)\Big) + \sum_{0 \leq u < t}\big(Z^T_u - Z^{T'}_u\big)^{\ast}M_{u+1}$$
is a martingale under $\Q_3^{T,T'}$. It follows by induction that
$$Y^T_t - Y^{T'}_t = (1 + \alpha)^{-(r-t)}\E_{\Q_3^{T,T'}}\big[Y^T_r - Y^{T'}_r \, \big| \, \F_t\big],$$
for all $t < r \leq T$. Then, by the boundedness established above,
$$\big|Y^T_t - Y^{T'}_t\big| \leq (1 + \alpha)^{-(T-t)}\E_{\Q_3^{T,T'}}\Big[\big|Y^{T'}_T\big| \, \Big| \, \F_t\Big] \leq \frac{C(1 + \alpha)^{-(T-t)}}{\alpha}.$$
Hence, we see that $\big\{Y^T_t\big\}_{T \geq t}$ is a Cauchy sequence in $T$ with respect to the $L^{\infty}$ norm. Therefore, for any $t \geq 0$, there exists an $\F_t$-measurable random variable $Y_t$, such that $Y^T_t \rightarrow Y_t$ in $L^{\infty}$ as $T \rightarrow \infty$, and this convergence is clearly uniform on finite sets in $t$. Further, $Y = \{Y_t\}_{t \geq 0}$ is also uniformly bounded by $C/\alpha$.

Taking an $\E_{\Prob}[\, \cdot \, | \, \F_t]$ expectation in (\ref{eq:OneStepTT'}), we deduce that
$$\big|\big(Z^T_t - Z^{T'}_t\big)^{\ast}M_{t+1}\big| \leq \big|Y^T_{t+1} - Y^{T'}_{t+1}\big| + \E_{\Prob}\Big[\big|Y^T_{t+1} - Y^{T'}_{t+1}\big| \, \Big| \, \F_t\Big],$$
and it follows that $\big\{\big(Z^T_t\big)^{\ast}M_{t+1}\big\}_{T>t}$ is a Cauchy sequence in $T$. Hence, there exists an $\F_{t+1}$-measurable random variable $U_{t+1}$ such that $\big(Z^T_t\big)^{\ast}M_{t+1} \rightarrow U_{t+1}$ in $L^{\infty}$ as $T \rightarrow \infty$. By Corollary~\ref{CorMartRepThm}, there exists an $\F_t$-measurable random variable $Z_t$ such that $U_{t+1} = Z_t^{\ast}M_{t+1}$.

Recall the one-step dynamics of (\ref{eq:DiscountedUpToTimeT}), given by
\begin{equation}\label{OneStepDiscounted}
Y^T_{t+1} = Y^T_t - \big(f(\omega,t,Z^T_t) - \alpha Y^T_t\big) + \big(Z^T_t\big)^{\ast}M_{t+1}.
\end{equation}
It follows from the Lipschitz condition that $f(\omega,t,Z^T_t) \rightarrow f(\omega,t,Z_t)$ in $L^{\infty}$ as $T \rightarrow \infty$. Finally, taking the limit as $T \rightarrow \infty$ in (\ref{OneStepDiscounted}), we deduce that
$$Y_{t+1} = Y_t - \big(f(\omega,t,Z_t) - \alpha Y_t\big) + Z_t^{\ast}M_{t+1},$$
and that this holds for every $t \geq 0$. Note that these are the one-step dynamics of (\ref{eq:DiscountedBSDE2}). Hence, $(Y,Z)$ is a bounded adapted solution of (\ref{eq:DiscountedBSDE2}), which we know is unique among such solutions.
\end{proof}

\subsection{$\gamma$-balanced drivers}

We now return to the case when our underlying process $X$ is a (possibly time-inhomogeneous) Markov chain under $\Prob$, with associated transition matrix $A_t$.

\begin{mydef}\label{DefnGammaControlled}
Let $A$ and $B$ be transition matrices and let $\gamma \in (0,1)$. We shall say that $B$ is $\gamma$-controlled by $A$ (or $B$ is controlled by $A$ with constant $\gamma$), and write $B \succeq_{\gamma} A$, if $B \geq \gamma A$ componentwise. If $B \succeq_{\gamma} A$ and $A \succeq_{\gamma} B$ then we will write $A \sim_{\gamma} B$.
\end{mydef}

\begin{myremark}
Writing $A = [a_{ij}]$ and $B = [b_{ij}]$, it is easy to check that the statement $B \succeq_{\gamma} A$ is equivalent to the condition
$$\gamma \leq \inf \left\{ \frac{b_{ij}}{a_{ij}} \ \middle| \ i,j \ \text{such that} \ a_{ij} \neq 0 \right\}.$$
We also note that $A \sim_{\gamma} B$ if and only if $A$ and $B$ have the same pattern of zero entries, and
$$\frac{b_{ij}}{a_{ij}} \in [\gamma,\gamma^{-1}]$$
for all $i,j$ such that $a_{ij} \neq 0$.
\end{myremark}

\begin{mydef}\label{DefnGammaBalanced}
We shall say that a driver $f$ is $\gamma$-balanced if for some $\gamma \in (0,1)$, there exists a random field $\psi : \Omega \times \mathbb{N} \times \Reals^N \times \Reals^N \rightarrow \Reals^N$ which is adapted (in the same sense as $f$), and such that for any $t \geq 0$, $y \in \Reals$, $z,z' \in \Reals^N$ and for almost all $\omega$, we have
\begin{enumerate}[(i)]
\item $f(\omega,t,y,z) - f(\omega,t,y,z') = (z-z')^{\ast}\big(\psi(\omega,t,z,z') - A_tX_t\big)$,
\item for each $e_i \in \States$,
$$\frac{e_i^{\ast}\psi(\omega,t,z,z')}{e_i^{\ast}A_tX_t} \in [\gamma,\gamma^{-1}],$$
where here $0/0 := 1$, and
\item $\mathbf{1}^{\ast}\psi(\omega,t,z,z') = 1$, where we recall that $\mathbf{1}$ denotes the vector in $\Reals^N$ with all entries equal to $1$.
\end{enumerate}
\end{mydef}

\begin{mylemma}\label{GammaBalLip}
If a driver $f$ is $\gamma$-balanced, then it is Lipschitz in $z$, with Lipschitz constant $1/\gamma$.
\end{mylemma}

\begin{proof}
Let $\psi = \psi(\omega,t,z,z')$ be the random field associated with $f$ in Definition~\ref{DefnGammaBalanced}. Let $t \geq 0$ and let $F$ be an atom of $\F_t$. Suppose that $X_t$ takes the value $e_j$ on $F$. Writing $A_t = [a_{ij}]$ and $\psi = [\psi_i]$, we note, from property (ii) of Definition~\ref{DefnGammaBalanced}, that $\psi_i \leq \gamma^{-1}a_{ij}$. Then, for any $\F_t$-measurable random variables $Y_t,Z_t,Z'_t$ evaluated on $F$, we have
\begin{align*}
\big| f(\omega&,t,Y_t,Z_t) - f(\omega,t,Y_t,Z'_t) \big| = \Bigg| \sum_{i=1}^N \psi_i (Z_t - Z'_t)^{\ast}(e_i - A_te_j) \Bigg| \\
&\leq \frac{1}{\gamma} \sum_{i=1}^N a_{ij} \, \big| (Z_t - Z'_t)^{\ast}(e_i - A_te_j) \big| = \frac{1}{\gamma} \, \E \Big[ \big| (Z_t - Z'_t)^{\ast}M_{t+1} \big| \, \Big| \, F \Big].
\end{align*}
This holds on every atom $F$ of $\F_t$, and the result follows.
\end{proof}

\begin{mylemma}\label{GammaBalPropMeasure}
If $f$ is a $\gamma$-balanced driver, then $f$ satisfies the conditions of Proposition~\ref{ChangeMeasure}.
\end{mylemma}

\begin{proof}
Let $\psi = \psi(\omega,t,z,z')$ be the random field associated with $f$ in Definition~\ref{DefnGammaBalanced}. Let $t \geq 0$, $y \in \Reals$, $z,z' \in \Reals^N$ and let $F$ be any atom of $\F_t$. Recall the notation $\J_t^F := \{i : \Prob(X_{t+1} = e_i \, | \, F) > 0\}$. Writing $\psi = [\psi_i]$, we see, from the properties of $\psi$ given in Definition~\ref{DefnGammaBalanced}, that $\psi_i > 0 \ \Longleftrightarrow \ e_i^{\ast}A_tX_t > 0 \ \Longleftrightarrow \ i \in \J_t^F$, and hence that $\sum_{i \in \J_t^F} \psi_i = 1$, and
\begin{equation}\label{eq:BalancedMarkovian}
f(\omega,t,y,z) - f(\omega,t,y,z') = \sum_{i \in \J_t^F} \psi_i(z-z')^{\ast}(e_i - \E[X_{t+1} \, | \, \F_t]).
\end{equation}
Write $S_t^F := \big\{(z-z')^{\ast}(e_i - \E[X_{t+1} \, | \, \F_t]) : i \in \J_t^F\big\}$, where the conditional expectation is evaluated on $F$. Note that
\begin{equation}\label{eq:CondExpZero2}
\E\big[(z - z')^{\ast}(X_{t+1} - \E[X_{t+1} \, | \, \F_t]) \, \big| \, F\big] = 0,
\end{equation}
and hence that $\min \big(S_t^F\big) \leq 0$ and $\max \big(S_t^F\big) \geq 0$. If $\min \big(S_t^F\big) < 0$, then (\ref{eq:BalancedMarkovian}) implies that
$$f(\omega,t,y,z) - f(\omega,t,y,z') > \min \big(S_t^F\big).$$
If $\min \big(S_t^F\big) = 0$, then (\ref{eq:CondExpZero2}) implies that $S_t^F = \{0\}$, and it clearly follows from (\ref{eq:BalancedMarkovian}) that
$$f(\omega,t,y,z) - f(\omega,t,y,z') = 0.$$
This establishes the conditions of Proposition~\ref{ChangeMeasure}.
\end{proof}

The following lemma is a discrete time version of \cite[Lemma 2]{Cohen2014}.

\begin{mylemma}\label{infGammaBalLemma}
Fix $\gamma \in (0,1)$. Let $\{f^u\}_{u \in \mathcal{U}}$ be a family of $\gamma$-balanced drivers that are independent of $y$, and for each $u \in \mathcal{U}$ let $\psi^u(\omega,t,z,z')$ be the random field associated with $f^u$. Let
$$g(\omega,t,z) := \inf_{u \in \mathcal{U}} \, \{f^u(\omega,t,z)\}.$$
Suppose $|g(\omega,t,z)| < \infty$ for all $t,z$ and almost all $\omega$. Then $g$ is also $\gamma$-balanced, and similarly for $\sup_{u \in \mathcal{U}} \{f^u\}$.
\end{mylemma}

\begin{proof}
By property (i) of Definition~\ref{DefnGammaBalanced}, the functions $\{f^u\}_{u \in \mathcal{U}}$ are continuous in $z$, uniformly in $u$. It follows that $g$ is a continuous, and hence measurable function of $z$.

Fix $t\ge 0$. As $\F_t$ can be made up from finitely many atoms, there exists a single null set $N_t \in \F_t$ such that $N' \subseteq N_t$ for every null set $N'\in \F_t$. For simplicity, define $g(\omega, t, z) = 0$ for $\omega \in N_t$. 

For every $\omega \not\in N_t$ and every $z,z' \in \mathbb{R}^N$, we can find a sequence $\{u_n\}_{n\geq 1}$ in $\mathcal{U}$ (dependent on $z,z'$) such that, omitting the arguments $\omega,t$ for clarity,
\begin{equation}\label{Convergetoinf}
\lim_{n \rightarrow \infty} (z-z')^{\ast}\psi^{u_n}(z,z') = \inf_{u \in \mathcal{U}} \big\{(z-z')^{\ast}\psi^u(z,z')\big\}.
\end{equation}
For each $n$, $\psi^{u_n}(z,z')$ is a stochastic vector, i.e. it belongs to the compact subset $\{\varphi \in \Reals^N : \varphi_i \geq 0, \, \sum_{i=1}^N \varphi_i = 1\}$. Hence, there exists a subsequence $\{u_{n_k}\}_{k \geq 1}$ such that
$$\lim_{k \rightarrow \infty} \psi^{u_{n_k}}(z,z') = \ubar{\psi}(z,z'),$$
where $\ubar{\psi}(z,z')$ is a stochastic vector which satisfies conditions (ii) and (iii) of Definition~\ref{DefnGammaBalanced}. In addition, $\ubar{\psi}$ can be constructed in such a way that it is measurable in $z, z'$ (see \cite[Theorem A.10.5]{CohenElliott2015}). Further, by (\ref{Convergetoinf}), we have
$$(z-z')^{\ast}\ubar{\psi}(z,z') = \inf_{u \in \mathcal{U}} \big\{(z-z')^{\ast}\psi^u(z,z')\big\}.$$
By a similar argument, there exists a stochastic vector $\bar{\psi}(z,z')$, which also satisfies conditions (ii) and (iii) of Definition~\ref{DefnGammaBalanced}, is measurable in $z, z'$, and is such that
$$(z-z')^{\ast}\bar{\psi}(z,z') = \sup_{u \in \mathcal{U}} \big\{(z-z')^{\ast}\psi^u(z,z')\big\}.$$
We have, for each $u \in \mathcal{U}$,
$$f^u(z) - f^u(z') = (z-z')^{\ast}\big(\psi^u(z,z') - A_tX_t\big),$$
from which we deduce that
\begin{align*}
(z-z')^{\ast}\ubar{\psi}(z,z') &\leq g(z) - g(z') + (z-z')^{\ast}A_tX_t \\
&\leq (z-z')^{\ast}\bar{\psi}(z,z').
\end{align*}
With $0/0 := 1/2$, it follows that
$$\mu := \frac{(z-z')^{\ast}\bar{\psi}(z,z') - \big(g(z) - g(z')\big) - (z-z')^{\ast}A_tX_t}{(z-z')^{\ast}\big(\bar{\psi}(z,z') - \ubar{\psi}(z,z')\big)} \in [0,1].$$
Define $\hat{\psi}(z,z') := \mu\ubar{\psi}(z,z') + (1-\mu)\bar{\psi}(z,z')$. Then
$$g(z) - g(z') = (z-z')^{\ast}\big(\hat{\psi}(z,z') - A_tX_t\big),$$
and we see that the random field $\hat{\psi}(\omega,t,z,z')$ satisfies all the conditions of Definition~\ref{DefnGammaBalanced} for the driver $g$, and hence that $g$ is $\gamma$-balanced.

The proof for $\sup_{u \in \mathcal{U}} \{f^u\}$ is similar.
\end{proof}

\subsection{Markovian solutions}

We showed in Lemma~\ref{FiniteHorizonMarkovSoln} that, if $X$, $f$ and the terminal condition $\xi$ are all Markovian, i.e. at time $t$ they only depend on $\omega$ through the value of $X_t$, then so is the solution $(Y,Z)$ of the corresponding finite horizon BSDE. We shall now extend this result to our infinite horizon equations. As previously, we assume throughout that $X$ is a Markov chain with transition matrix $A_t$.

\begin{mylemma}\label{MarkovianSolnsDisc}
Let $\alpha > 0$, let $f$ be a Markovian driver satisfying the conditions of Theorem~\ref{SolnDiscountedBSDE}, and let $(Y,Z)$ be the unique bounded solution of the associated discounted BSDE. Then there exists a deterministic function $v : \mathbb{N} \times \States \rightarrow \Reals$ such that, for all $t$, $Y_t = v(t,X_t)$ and $e_k^{\ast}Z_t = v(t+1,e_k)$ for all $e_k \in \States$.

Moreover, writing $\mathbf{v}_{t+1}$ for the vector in $\Reals^N$ with entries $e_k^{\ast}\mathbf{v}_{t+1} = v(t+1,e_k)$, the function $v$ satisfies
\begin{equation}\label{eq:veqn2}
(1 + \alpha)v(t,e_k) - \tilde{f}(e_k,t,\mathbf{v}_{t+1}) - \mathbf{v}_{t+1}^{\ast}A_te_k = 0
\end{equation}
for all $t$ and all $e_k \in \States$ such that $\Prob(X_t = e_k) > 0$.
\end{mylemma}

\begin{proof}
Let $(Y^T,Z^T)$ denote the (unique adapted) solution of the finite horizon BSDE (\ref{eq:DiscountedUpToTimeT}). Note that we are in the situation of Lemma~\ref{FiniteHorizonMarkovSoln}, with $\varphi \equiv 0$. Hence, for each $T \geq 0$, there exists a function $v^T : \{0,1,\ldots,T\} \times \States \rightarrow \Reals$ such that $Y^T_t = v^T(t,X_t)$ and $e_k^{\ast}Z^T_t = v^T(t+1,e_k)$ for all $e_k \in \States$.

Recall, from Theorem~\ref{SolnDiscountedBSDE} that $Y^T_t \rightarrow Y_t$ as $T \rightarrow \infty$ with respect to the $L^{\infty}$ norm, uniformly on finite sets in $t$, from which we deduce the existence of a function $v : \mathbb{N} \times \States \rightarrow \Reals$ such that $Y_t = v(t,X_t)$ for all $t$.

Note that the one-step dynamics of the infinite horizon BSDE (\ref{eq:DiscountedBSDE2}) are identical to the one-step dynamics of the corresponding finite horizon BSDE (\ref{eq:DiscountedUpToTimeT}). Therefore, given the value of $Y_{t+1}$, we can construct the value of $Z_t$ exactly as in the finite horizon case. In fact, since $Y_{t+1}$ is Markovian, we can apply the same argument as in the proof of Lemma~\ref{FiniteHorizonMarkovSoln} to deduce that $e_k^{\ast}Z_t = v(t+1,e_k)$ for all $e_k \in \States$. The equation (\ref{eq:veqn2}) follows from the one-step dynamics of (\ref{eq:DiscountedBSDE2}).
\end{proof}

\begin{mycor}\label{vIndepoft}
Recall the conditions of Lemma~\ref{MarkovianSolnsDisc}. Suppose that $X$ is time-homogeneous, so that $A = A_t$ does not depend on $t$, and that the driver $f = \tilde{f}(X_t,z)$ is also independent of $t$. Then the same is true of the function $v$ given in Lemma~\ref{MarkovianSolnsDisc}.
\end{mycor}

\begin{proof}
The discounted BSDE (\ref{eq:DiscountedBSDE2}) is given in this case by
$$Y_T = Y_t - \sum_{t \leq u < T} \big( \tilde{f}(X_u,Z_u) - \alpha Y_u \big) + \sum_{t \leq u < T} Z_u^{\ast}(X_{u+1} - AX_u).$$
Note that, given $X_u$, the summands are independent of $u$. By the uniqueness of solutions established in Theorem~\ref{SolnDiscountedBSDE}, the same must therefore be true of the solution $(Y_u,Z_u)$, and the result follows.
\end{proof}

\section{Uniformly Ergodic Markov Chains}

The main aim of this section is to obtain ergodicity estimates for a discrete time Markov chain which hold uniformly for a suitable class of perturbations of its transition matrix. We will loosely follow the argument given in Section 3 of \cite{CohenHu2013}, by adapting the main steps into the discrete time setting. Henceforth, we will assume that all the Markov chains we consider are time-homogeneous.

\subsection{Uniform ergodicity}

The following lemma will turn out to be a useful tool for improving bounds on expectations, and improves slightly the estimate given in \cite[Lemma 3.3]{CohenHu2013}.

\begin{mylemma}\label{ImproveBound}
Let $T$ be a random variable, and consider $G(\beta) = \sup_{\nu} \E_{\nu} \big[ e^{\beta T} \big]$, where $\nu$ is a parametrisation of probability measures under which the expectation is taken. Suppose there exist finite constants $\beta^{\ast}, K > 0$ such that $G(\beta^{\ast}) \leq K$. Then, for any $\epsilon > 0$,
$$G(\beta) \leq 1 + \epsilon \ \ \ \ \text{for all} \ \ \ \ \beta \in \left[ 0, \beta^{\ast} \Big( \frac{\epsilon}{K} \wedge 1 \Big) \right].$$
\end{mylemma}

\begin{proof}
Let $0 \leq c \leq \left( \frac{\epsilon}{K} \wedge 1 \right)$, so that $e^{c \beta^{\ast} T} \leq 1 + ce^{\beta^{\ast} T}$ for any $T \geq 0$. Then, for any given measure $\nu$ in the class considered, we have
$$\E_{\nu} \big[ e^{c \beta^{\ast} T} \big] \leq \E_{\nu} \big[ \Ind_{\{ T<0 \}} + \Ind_{\{ T \geq 0 \}} \big( 1 + ce^{\beta^{\ast} T} \big) \big] \leq 1 + c\E_{\nu} \big[ e^{\beta^{\ast} T} \big] \leq 1 + \epsilon.$$
Taking the supremum over all measures $\nu$, we deduce the result.
\end{proof}

\begin{mydef}\label{DefnMeasuresTVnorm}
Let $\tilde{\mathcal{M}}(\States)$ denote the vector space (over $\mathbb{R}$) of finite signed measures on $\left( \States, \mathcal{P}(\States) \right)$. We will endow this space with the total variation norm, given by
$$\| \mu \|_{TV} := \frac{1}{2} \sum_{x \in \States} \big| \mu (\{x\}) \big|,$$
for $\mu \in \tilde{\mathcal{M}}(\States)$. By standard results, $\tilde{\mathcal{M}}(\States)$ is a Banach space. Let $\mathcal{M}$ denote the set of probability measures on $\left( \States, \mathcal{P}(\States) \right)$. Note that $\mathcal{M}$ is a closed subset of $\tilde{\mathcal{M}}(\States)$.
\end{mydef}

\begin{mydef}\label{DefnUniformErgodicity}
Let $X$ be a Markov chain on $\States$, and $P_t$ its transition operator, so that $P_t \mu$ is the law of $X_t$ given $X_0 \sim \mu$. We say that the chain $X$ is \textit{uniformly ergodic} if there exists a probability measure $\pi$ on $\States$, and constants $R, \rho > 0$ such that
$$\sup_{\mu \in \mathcal{M}} \| P_t \mu - \pi \|_{TV} \leq Re^{-\rho t} \ \ \ \ \text{for all} \ \ \ \ t \geq 0.$$
In this case $\pi$ is the unique invariant measure for $X$.
\end{mydef}

Under our assumptions of discrete time chains on a finite state space, we have the following simple classification of uniformly ergodic chains.

\begin{myprop}\label{UniErgEquivalence}
Let $X$ and $Y$ be two independent copies of a Markov chain on the finite state space $\States$, and let $T = \inf\{t \geq 1 : X_t = Y_t\}$ be the first meeting time of the two chains. Then the following statements are equivalent:
\begin{enumerate}[(i)]
\item $X$ is uniformly ergodic.
\item $X$ is irreducible up to transient states (i.e. $X$ has only one closed communicating class) and is aperiodic.
\item There exists $\beta > 0$ such that
$$G^{\ast}(\beta) := \sup_{x,y \in \States} \E_{xy} \big[ e^{\beta T} \big] < \infty,$$
where $\E_{xy}$ denotes expectation conditional on $X_0 = x$ and $Y_0 = y$.
\end{enumerate}
\end{myprop}

\begin{proof}
The implications (\textit{i}) $\Rightarrow$ (\textit{ii}) $\Rightarrow$ (\textit{iii}) are straightforward. The implication (\textit{iii}) $\Rightarrow$ (\textit{i}) will be demonstrated in the course of the proof of Theorem~\ref{BisUniformlyErgodic}.
\end{proof}

Note that, by Lemma~\ref{ImproveBound}, we can make $G^{\ast}(\beta)$ arbitrarily close to $1$ by choosing $\beta$ sufficiently small.

\subsection{Split chains}

Recall from Definition~\ref{DefnGammaControlled}, that we say $B$ is $\gamma$-controlled by $A$, and write $B \succeq_{\gamma} A$, if $B \geq \gamma A$ componentwise, and we write $A \sim_{\gamma} B$ if both $B \succeq_{\gamma} A$ and $A \succeq_{\gamma} B$.

\begin{mythm}\label{BisUniformlyErgodic}
Let $A$ and $B$ be transition matrices such that $B$ is $\gamma$-controlled by $A$ for some $\gamma \in (0,1)$. If, under the measure induced by $A$, the chain $X$ is uniformly ergodic, then $X$ is also uniformly ergodic under the measure induced by $B$. Furthermore, the constants $R$ and $\rho$ in Definition~\ref{DefnUniformErgodicity} depend only on $A$ and $\gamma$, and $R$ can be made arbitrarily close to $1$ (with a corresponding decrease in $\rho$).
\end{mythm}

In our discrete time, finite state setting, the fact that $X$ is uniformly ergodic under the measure induced by $B$ is trivial. The interesting part of Theorem~\ref{BisUniformlyErgodic} is the claim there exist constants of ergodicity $R,\rho$ which hold for \emph{all} transition matrices $B$ such that $B \succeq_{\gamma} A$. We shall make use of this result in Section 4, in the proof of existence of bounded Markovian solutions to Ergodic BSDEs.

The proof of this theorem is the main purpose of this section. To do this we will use a Nummelin splitting argument, the broad theory of which is described in Meyn and Tweedie \cite[Chapter 5]{MeynTweedie2009}. We will assume the conditions of the theorem throughout the remainder of the section.

\begin{mydef}\label{DefnSplit}
Define the \textit{split space} of $\States$ to be $\check{\States} := \States \times \{0,1\}$. We will also denote the \textit{layers} of the splitting by $\States_0 := \States \times \{0\}$ and $\States_1 := \States \times \{1\}$, so that $\check{\States} = \States_0 \cup \States_1$. Given $y \in \States$, denote the corresponding elements in $\States_0$ and $\States_1$ by $\check{y}_0$ and $\check{y}_1$ respectively.

Let $\nu \in \tilde{\mathcal{M}}(\States)$. For each $y \in \States$, let
$$\check{\nu} (\{\check{y}_0\}) = (1 - \gamma) \nu (\{y\}), \ \ \ \ \ \ \check{\nu} (\{\check{y}_1\}) = \gamma \nu (\{y\}).$$
This defines a split measure $\check{\nu}$ on $\left( \check{\States}, \mathcal{P}(\check{\States}) \right)$. Similarly, we define the splitting of a column vector $\phi \in \mathbb{R}^N$ via the splitting map $\check{\Pi} : \mathbb{R}^N \rightarrow \mathbb{R}^{2N}$, defined by
$$\check{\Pi}(\phi) := (1 - \gamma) \left[ \begin{array}{c}
\phi \\
0
\end{array}
\right] + \gamma \left[ \begin{array}{c}
0 \\
\phi
\end{array}
\right] = \left[ \begin{array}{c}
\left( 1 - \gamma \right) \phi \\
\gamma \phi
\end{array}
\right] \in \mathbb{R}^{2N},$$
where here $0$ is the zero vector in $\mathbb{R}^N$.
\end{mydef}

\begin{myremark}\label{MeasuresAndVectors}
Note that these splittings are consistent in the following way. Consider a measure $\nu$ as a vector in $\mathbb{R}^N$, so that $\nu(\{y\}) = y^{\ast}\nu$ for each $y \in \States$. Similarly, consider the split measure $\check{\nu}$ as a vector in $\mathbb{R}^{2N}$, so that
$$\check{\nu} (\{ \check{y}_0 \}) = \left[ \begin{array}{c}
y \\
0
\end{array}
\right]^{\ast} \check{\nu}, \ \ \ \ \ \ \check{\nu} (\{ \check{y}_1 \}) = \left[ \begin{array}{c}
0 \\
y
\end{array}
\right]^{\ast} \check{\nu}$$
for each $y \in \States$. Then we see that the split measure $\check{\nu}$ corresponds precisely to the split vector $\check{\Pi} (\nu)$.

Notice that if $\nu$ is a probability measure on $\States$, then $\check{\nu}$ is a probability measure on $\check{\States}$. Equivalently, if $\phi$ is a stochastic vector in $\mathbb{R}^N$, then $\check{\Pi} (\phi)$ is a stochastic vector in $\mathbb{R}^{2N}$.
\end{myremark}

\begin{mydef}
We can define the splitting of a matrix by splitting its columns as in the previous definition. That is, given a matrix $M$, define
$$\check{M} := \left[ \begin{array}{c}
\left( 1 - \gamma \right) M \\
\gamma M
\end{array}
\right].$$
\end{mydef}

\begin{mydef}\label{DefnSplitMatrix}
Recall the assumptions of Theorem~\ref{BisUniformlyErgodic}, in particular that $A$ and $B$ are transition matrices with $B \succeq_{\gamma} A$ for some $\gamma \in (0,1)$. Let
$$C = \left( 1 - \gamma \right)^{-1} \left( B - \gamma A \right),$$
and let $\mathcal{B} = \left[ \, \check{C} \, \middle| \, \check{A} \, \right]$, i.e. the augmented matrix formed from $\check{C}$ and $\check{A}$. Note that $C$ and $\mathcal{B}$ are both transition matrices.
\end{mydef}

We can now define a new Markov chain $\check{X}$ on the split space $\check{\States}$, which jumps according to transition matrix $\mathcal{B}$. Intuitively, transitions occur from a state $\check{x} \in \States_0$ according to the vector $Cx$, and from a state $\check{x} \in \States_1$ according to the vector $Ax$ (where $x$ is the projection of $\check{x}$ in $\States$), except that the result is randomly split between the layers $\States_0$ and $\States_1$ with probabilities $1 - \gamma$ and $\gamma$ respectively.

\begin{mylemma}\label{CommutePiwithB}
Let $\phi \in \mathbb{R}^N$. Then $\check{\Pi}(B\phi) = \mathcal{B} \, \check{\Pi}(\phi)$.
\end{mylemma}

\begin{proof}
By the definitions of $\check{\Pi}$, $\mathcal{B}$ and $C$, we have
\begin{align*}
\mathcal{B} \, \check{\Pi}(\phi) &= (1 - \gamma) \check{C} \phi + \gamma \check{A} \phi \\
&= (1 - \gamma) \check{\Pi} (C\phi) + \gamma \check{\Pi} (A\phi) = \check{\Pi} (B\phi).
\end{align*}
\end{proof}

\begin{myprop}\label{MarginalDistributions}
Let $X$ be a Markov chain on $\States$ with initial distribution $X_0 \sim \nu$ and transition matrix $B$, and let $\check{X}$ be a Markov chain on $\check{\States}$ with initial distribution $\check{X}_0 \sim \check{\nu}$ and transition matrix $\mathcal{B}$. Then $X$ and $\check{X}$ have the same marginal distributions, ignoring the splitting. That is, for any $t$ and any $y \in \States$,
$$y^{\ast} \E [X_t] = \Prob (X_t = y) = \Prob (\check{X}_t \in \{\check{y}_0, \check{y}_1\}).$$
\end{myprop}

\begin{proof}
Consider $\nu$ as a vector in $\mathbb{R}^N$ and $\check{\nu}$ as a vector in $\mathbb{R}^{2N}$, as in Remark~\ref{MeasuresAndVectors}. By the Chapman--Kolmogorov equation, the distribution of $X_t$ is given by $\E [X_t] = B^t\nu$, and the distribution of $\check{X}_t$ is given by $\E [\check{X}_t] = \mathcal{B}^t\check{\nu}$. By applying Lemma~\ref{CommutePiwithB} $t$ times, we obtain $\check{\Pi} (B^t\nu) = \mathcal{B}^t \, \check{\Pi} (\nu)$. Recall from Remark~\ref{MeasuresAndVectors} that $\check{\Pi} (\nu) = \check{\nu}$. Hence, we have $\check{\Pi} (\E [X_t]) = \E [\check{X}_t]$, and the result follows.
\end{proof}

The following lemma follows easily from the definition of $\mathcal{B}$ and the basic notion of conditional probability.

\begin{mylemma}\label{StayInLayerOne}
Let $\check{X}$ be as in the above proposition. Write $\E^{\ast}$ for the expectation conditioned on $\check{X}$ never leaving the layer it starts in (i.e. never jumps from $\States_0$ to $\States_1$, or from $\States_1$ to $\States_0$). Then, under $\E^{\ast}$, $\check{X}$ jumps according to the transition matrix
$$\left[ \begin{array}{cc}
C & 0 \\
0 & A
\end{array}
\right].$$
\end{mylemma}

Intuitively, under $\E^{\ast}$, $\check{X}$ either jumps in $\States_0$ following transition matrix $C$, or in $\States_1$ following transition matrix $A$, depending on which layer it starts in.

\subsection{Exponential moment bounds}

We will now consider two independent copies of our Markov chain on the split space, and show that the first meeting time of these chains admits exponential moments.

\begin{mythm}\label{ExpMoments}
Let $\check{X}$ and $\check{Y}$ be two independent Markov chains on the split space $\check{\States}$, each following transition matrix $\mathcal{B}$, as defined in Definition~\ref{DefnSplitMatrix}. Let $\check{S} := \inf \{t \geq 0 : \check{X}_t = \check{Y}_t\}$. Then, for any $\epsilon > 0$, there exists $\tilde{\beta} > 0$ such that
$$H^{\ast}(\beta) := \sup_{\check{x}, \check{y} \in \check{\States}} \E_{\check{x}\check{y}} \big[ e^{\beta \check{S}} \big] \leq 1 + \epsilon \ \ \ \ \text{for all} \ \ \ \ \beta \in \big[ 0,\tilde{\beta} \big],$$
where $\E_{\check{x}\check{y}}$ is the expectation conditional on $\check{X}_0 = \check{x}$ and $\check{Y}_0 = \check{y}$. Furthermore, $\tilde{\beta}$ does not depend on $B$ except through $A$ and $\gamma$.
\end{mythm}

The proof of this theorem will be done in a number of steps. First, we will introduce some more notation.

Let $K_t$ denote the number of jumps, up to time $t$, that result in both $\check{X}$ and $\check{Y}$ being in $\States_1$ when they were not both in $\States_1$ previously, and denote by $t_k$ the time of the $k$\textsuperscript{th} such transition. That is, let $t_0 = 0$ and
$$t_k = \inf \{t > t_{k-1} : (\check{X}_t, \check{Y}_t \in \States_1) \setminus (\check{X}_{t-1}, \check{Y}_{t-1} \in \States_1)\}$$
for $k \geq 1$. Finally, let
$$\check{T} = \inf \{t \geq 1 : (\check{X}_t = \check{Y}_t \in \States_1) \cap (\check{X}_{t-1}, \check{Y}_{t-1} \in \States_1)\},$$
i.e. the first time the chains meet in $\States_1$, both having jumped from a state in $\States_1$.

\begin{mylemma}\label{BoundforExpMoment}
Let $G^{\ast}(\beta)$ be the function defined in Proposition~\ref{UniErgEquivalence}, that is, the supremum over starting states of the moment generating function of the first meeting time on the basic (unsplit) state space $\States$ of two independent Markov chains with transition matrix $A$. Then, for any $\check{x},\check{y} \in \check{\States}$,
\begin{equation}\label{eq:BoundExpMoment}
\E_{\check{x}\check{y}} \big[ e^{\beta \check{T}} \big] \leq G^{\ast}(2\beta)^{\frac{1}{2}} \sum_{k=0}^{\infty} \E_{\check{x}\check{y}} \big[ e^{2\beta t_k} \big]^{\frac{1}{2}} \Prob_{\check{x}\check{y}} (K_{\check{T}} = k)^{\frac{1}{2}},
\end{equation}
where $\Prob_{\check{x}\check{y}}$ is the probability conditioned on $\check{X}_0 = \check{x}$ and $\check{Y}_0 = \check{y}$.
\end{mylemma}

\begin{proof}
Note that, conditioned on $\{K_{\check{T}} = k\}$, $\check{X}$ and $\check{Y}$ do not leave $\States_1$ between $t_k$, the $k$\textsuperscript{th} time they arrive in $\States_1$ when they were not both in $\States_1$ previously, and $\check{T}$, the time at which they next meet in $\States_1$. Write $\E^{\ast}$ for the expectation conditioned on both $\check{X}$ and $\check{Y}$ not leaving the layers they start in. Then
$$\E_{\check{x}\check{y}} \big[ e^{2\beta (\check{T} - t_k)} \, \big| \, K_{\check{T}} = k \big] \leq \sup_{x,y \in \States_1} \E^{\ast}_{xy} \big[ e^{2\beta \check{T}} \big] = G^{\ast}(2\beta),$$
since, by Lemma~\ref{StayInLayerOne}, under $\E^{\ast}$, $\check{X}$ and $\check{Y}$ jump in $\States_1$ following transition matrix $A$. Then
\begin{align*}
\E_{\check{x}\check{y}} \big[ e^{\beta \check{T}} \Ind_{\{K_{\check{T}} = k\}} \big] &= \E_{\check{x}\check{y}} \Big[ e^{\beta t_k} \E_{\check{x}\check{y}} \big[ e^{\beta (\check{T} - t_k)} \Ind_{\{K_{\check{T}} = k\}} \, \big| \, t_k \big] \Big] \\
&\leq \E_{\check{x}\check{y}} \big[ e^{2\beta t_k} \big]^{\frac{1}{2}} \E_{\check{x}\check{y}} \bigg[ \left( \E_{\check{x}\check{y}} \big[ e^{\beta (\check{T} - t_k)} \Ind_{\{K_{\check{T}} = k\}} \, \big| \, t_k \big] \right)^2 \bigg]^{\frac{1}{2}} \\
&\leq \E_{\check{x}\check{y}} \big[ e^{2\beta t_k} \big]^{\frac{1}{2}} \E_{\check{x}\check{y}} \Big[ \E_{\check{x}\check{y}} \big[ e^{2\beta (\check{T} - t_k)} \Ind_{\{K_{\check{T}} = k\}} \, \big| \, t_k \big] \Big]^{\frac{1}{2}} \\
&= \E_{\check{x}\check{y}} \big[ e^{2\beta t_k} \big]^{\frac{1}{2}} \left( \E_{\check{x}\check{y}} \big[ e^{2\beta (\check{T} - t_k)} \, \big| \, K_{\check{T}} = k \big] \Prob_{\check{x}\check{y}} (K_{\check{T}} = k) \right)^{\frac{1}{2}} \\
&\leq \E_{\check{x}\check{y}} \big[ e^{2\beta t_k} \big]^{\frac{1}{2}} \big( G^{\ast}(2\beta) \, \Prob_{\check{x}\check{y}} (K_{\check{T}} = k) \big)^{\frac{1}{2}},
\end{align*}
and the result follows.
\end{proof}

We shall now seek to bound the components of the sum in Lemma~\ref{BoundforExpMoment}.

\begin{mylemma}\label{BoundOnProb}
For states $x,y \in \States$, consider a pair of independent basic (unsplit) chains $X,Y$ following transition matrix $A$, with starting values $X_0 = x$ and $Y_0 = y$. For $n \geq 1$, let $q(x,y;n)$ denote the probability that the first positive time the chains meet is at time $n$. Let
$$q_{\gamma} := \inf_{x,y \in \States} \sum_{n=1}^{\infty} q(x,y;n) \gamma^{2n},$$
where we recall that $\gamma \in (0,1)$. Then $q_{\gamma} \in (0,1)$ and, for any $\check{x},\check{y} \in \check{\States}$ and any $k \geq 0$,
$$\Prob_{\check{x}\check{y}} (K_{\check{T}} = k) \leq \gamma^2 (1 - q_{\gamma})^{k-1}.$$
\end{mylemma}

\begin{proof}
The fact that $q_{\gamma} \in (0,1)$ is clear, as we are working with a finite state space, and the independent chains almost surely meet after a finite time.

Now, suppose first that $\check{x},\check{y} \in \States_1$. Then, under $\Prob_{\check{x}\check{y}}$, $\{K_{\check{T}} = 0\}$ is the event that the split chains $\check{X}$ and $\check{Y}$ meet before either of them leaves $\States_1$. As they are independent, at each transition the probability that both chains will jump to $\States_1$ is $\gamma^2$. Further, by Lemma~\ref{StayInLayerOne}, for as long as the chains remain in $\States_1$, they jump following transition matrix $A$. Hence,
$$\Prob_{\check{x}\check{y}} (K_{\check{T}} = 0) = \sum_{n=1}^{\infty} q(x,y;n) \gamma^{2n} \leq \gamma^2.$$
Now let $k \geq 1$. Each time the chains both arrive in $\States_1$ when they were not both in $\States_1$ previously, they will either meet before one of them leaves $\States_1$, or they will not. Considering these as successes and failures, a geometric trials argument yields
\begin{align*}
\Prob_{\check{x}\check{y}} (K_{\check{T}} = k) &= \E_{\check{x}\check{y}} \Bigg[ \Bigg( \sum_{n=1}^{\infty} q(\check{X}_{t_k},\check{Y}_{t_k};n) \gamma^{2n} \Bigg) \prod_{i=0}^{k-1} \Bigg( 1 - \sum_{n=1}^{\infty} q(\check{X}_{t_i},\check{Y}_{t_i};n) \gamma^{2n} \Bigg) \Bigg] \\
&\leq \gamma^2 (1 - q_{\gamma})^k,
\end{align*}
where, in the notation of Definition~\ref{DefnSplit}, we define $q(\check{x}_1,\check{y}_1;n) = q(x,y;n)$ for notational simplicity. Now suppose that either $\check{x} \in \States_0$ or $\check{y} \in \States_0$. Clearly the chains must first both arrive in $\States_1$ before they can meet in $\States_1$, so $\Prob_{\check{x}\check{y}} (K_{\check{T}} = 0) = 0$. For $k \geq 1$, we have
\begin{align*}
\Prob_{\check{x}\check{y}} (K_{\check{T}} = k) &= \E_{\check{x}\check{y}} \Bigg[ \Bigg( \sum_{n=1}^{\infty} q(\check{X}_{t_k},\check{Y}_{t_k};n) \gamma^{2n} \Bigg) \prod_{i=1}^{k-1} \Bigg( 1 - \sum_{n=1}^{\infty} q(\check{X}_{t_i},\check{Y}_{t_i};n) \gamma^{2n} \Bigg) \Bigg] \\
&\leq \gamma^2 (1 - q_{\gamma})^{k-1}.
\end{align*}
Putting this together, we deduce the result.
\end{proof}

\begin{mylemma}\label{BoundOnMGF}
Let $q_{\gamma}$ be as in the previous lemma. Then there exists $\beta > 0$ such that, for any $\check{x},\check{y} \in \check{\States}$ and any $k \geq 0$,
$$\E_{\check{x}\check{y}} \big[ e^{2\beta t_k} \big] \leq \left( \frac{1}{2} \left( 1 + \frac{1}{1 - q_{\gamma}} \right) \right)^k.$$
Further, this $\beta$ only depends on $\gamma$ and $q_{\gamma}$.
\end{mylemma}

\begin{proof}
It will be useful to consider two independent geometric random variables
$$Q_1 \sim \text{Geom}(1 - \gamma^2) \ \ \ \ \text{and} \ \ \ \ Q_2 \sim \text{Geom}(\gamma^2),$$
where we adopt the definition of the probability mass function of $Q \sim \text{Geom}(p)$ as being given by $\Prob (Q=k) = (1-p)^{k-1}p$ for $k \geq 1$.

If $\check{x},\check{y} \in \States_1$ then $t_1$ is the time taken for either of the chains to leave $\States_1$, plus the time taken for both chains to return to $\States_1$. It follows from an application of the strong Markov property that these times are independent. We therefore have that $t_1 \stackrel{\text{d}}{=} Q_1 + Q_2$. If either $\check{x} \in \States_0$ or $\check{y} \in \States_0$ then $t_1$ is just the time taken for both chains to arrive in $\States_1$, so in this case $t_1 \stackrel{\text{d}}{=} Q_2$. Similarly, for $k \geq 2$, $t_k - t_{k-1}$ is the time taken for either of the chains to leave $\States_1$, plus the time taken for both chains to return to $\States_1$, and these times are independent, so it follows that $t_k - t_{k-1} \stackrel{\text{d}}{=} Q_1 + Q_2$. Hence, for all $k \geq 1$,
$$\E_{\check{x}\check{y}} \big[ e^{2\beta (t_k - t_{k-1})} \big] \leq \E \big[ e^{2\beta (Q_1 + Q_2)} \big].$$
By Lemma~\ref{ImproveBound}, we can make a new choice of $\beta > 0$, which depends only on $\gamma$ and $q_{\gamma}$, such that
$$\E \big[ e^{2\beta (Q_1 + Q_2)} \big] \leq \frac{1}{2} \bigg( 1 + \frac{1}{1 - q_{\gamma}} \bigg).$$

It follows from another application of the strong Markov property that $t_k - t_{k-1}$ is independent of $t_1, \ldots, t_{k-1}$ for each $k$. Then, for $k \geq 1$,
$$\E_{\check{x}\check{y}} \big[ e^{2\beta t_k} \big] = \prod_{j=1}^k \E_{\check{x}\check{y}} \big[ e^{2\beta (t_j - t_{j-1})} \big] \leq \E \big[ e^{2\beta (Q_1 + Q_2)} \big]^k \leq \bigg( \frac{1}{2} \bigg( 1 + \frac{1}{1 - q_{\gamma}} \bigg) \bigg)^k.$$
\end{proof}

\begin{proof}[Proof of Theorem~\ref{ExpMoments}.]
Substituting the bounds given by Lemmas~\ref{BoundOnProb} and \ref{BoundOnMGF} into (\ref{eq:BoundExpMoment}), we obtain
\begin{align*}
\E_{\check{x}\check{y}} \big[ e^{\beta \check{T}} \big] &\leq G^{\ast}(2\beta)^{\frac{1}{2}} \sum_{k=0}^{\infty} \left[ \left( \frac{1}{2} \left( 1 + \frac{1}{1 - q_{\gamma}} \right) \right)^k \cdot \gamma^2 (1 - q_{\gamma})^{k-1} \right]^{\frac{1}{2}} \\
&= \gamma \left( \frac{G^{\ast}(2\beta)}{1 - q_{\gamma}} \right)^{\frac{1}{2}} \left( 1 - \left( 1 - \frac{q_{\gamma}}{2} \right)^{\frac{1}{2}} \right)^{-1}.
\end{align*}
This is valid provided that $\beta$ is chosen to be at least as small as the $\beta$ given in Lemma~\ref{BoundOnMGF}. The expression above is finite provided that $\beta$ is also sufficiently small to ensure that $G^{\ast}(2\beta)$ is finite, which can be guaranteed by Proposition~\ref{UniErgEquivalence}. Note that this choice of $\beta$ depends only on $A$, $\gamma$ and $q_{\gamma}$.

With this $\beta$, we have
$$H^{\ast}(\beta) = \sup_{\check{x}, \check{y} \in \check{\States}} \E_{\check{x}\check{y}} \big[ e^{\beta \check{S}} \big] \leq \sup_{\check{x}, \check{y} \in \check{\States}} \E_{\check{x}\check{y}} \big[ e^{\beta \check{T}} \big] < \infty.$$
By Lemma~\ref{ImproveBound}, for any $\epsilon > 0$ there exists $\tilde{\beta} > 0$ such that $H^{\ast}(\beta) \leq 1 + \epsilon$ for all $\beta \in \big[ 0,\tilde{\beta} \big]$. Finally, we see that $\tilde{\beta}$ depends only on $A$, $\gamma$, $\epsilon$, and on $q_{\gamma}$, which is itself a function of $A$ and $\gamma$. In particular, $\tilde{\beta}$ does not depend on $B$ except through $A$ and $\gamma$.
\end{proof}

\begin{mycor}\label{BoundUnsplitChains}
Let $X$ and $Y$ be two independent copies of the Markov chain on $\States$ with transition matrix $B$. Let $T = \inf \{t \geq 0 : X_t = Y_t\}$ be the first meeting time of these chains. Then, for any $\epsilon > 0$, there exists $\tilde{\beta} > 0$ such that
$$\sup_{x,y \in \States} \E_{xy} \big[ e^{\beta T} \big] \leq 1 + \epsilon \ \ \ \ \text{for all} \ \ \ \ \beta \in \big[ 0,\tilde{\beta} \big],$$
where, as usual, $\E_{xy}$ is the expectation conditional on $X_0 = x$ and $Y_0 = y$. Furthermore, $\tilde{\beta}$ does not depend on $B$ except through $A$ and $\gamma$.
\end{mycor}

\begin{proof}
By Proposition~\ref{MarginalDistributions}, $X$ has the same marginal distribution as $\check{X}$, ignoring the splitting, and similarly for $Y$. Hence $T$, the first meeting time of $X$ and $Y$, is less than or equal (in distribution) to $\check{S}$, the first meeting time of $\check{X}$ and $\check{Y}$, which we showed has the required bound in Theorem~\ref{ExpMoments}.
\end{proof}

\begin{proof}[Proof of Theorem~\ref{BisUniformlyErgodic}.]
Let $X$ and $Y$ be two independent copies of the Markov chain on $\States$ with transition matrix $B$. Let $T = \inf \{t \geq 0 : X_t = Y_t\}$ be the first meeting time of these chains, as in the previous Corollary. Let $\mu, \nu \in \mathcal{M}$ (recall Definition~\ref{DefnMeasuresTVnorm}). Write $P_t$ for the transition operator of $X$ (and $Y$), so that $P_t \mu$ is the law of $X_t$ given $X_0 \sim \mu$, and $P_t \nu$ is the law of $Y_t$ given $Y_0 \sim \nu$. It is easy to see that, conditioned on $\{T \leq t\}$, $X_t$ and $Y_t$ have the same distribution. It follows that
$$\| P_t \mu - P_t \nu \| _{TV} \leq \Prob(T > t \, | \, X_0 \sim \mu, Y_0 \sim \nu).$$
By Corollary~\ref{BoundUnsplitChains}, for any $\epsilon > 0$, there exists $\tilde{\beta} > 0$, which depends only on $A$, $\gamma$ and $\epsilon$, such that
$$\sup_{x,y \in \States} \E_{xy} \big[ e^{\tilde{\beta} T} \big] \leq 1 + \epsilon.$$
Then, by Markov's inequality,
\begin{align}
\| P_t \mu - P_t \nu \|_{TV} &\leq \Prob(T > t \, | \, X_0 \sim \mu, Y_0 \sim \nu) \nonumber \\
&\leq \E \big[ e^{\tilde{\beta} T} \, \big| \, X_0 \sim \mu, Y_0 \sim \nu \big] e^{-\tilde{\beta} t} \leq (1 + \epsilon) e^{-\tilde{\beta} t}. \label{eq:TVconverge}
\end{align}
Let $s > t$. Replacing $\nu$ by $P_{s-t} \mu$ in the above, we obtain
$$\| P_t \mu - P_s \mu \|_{TV} \leq (1 + \epsilon) e^{-\tilde{\beta} t}.$$
Since this holds for all $s > t$, we see that $\{P_t \mu\}_{t \geq 0}$ is a Cauchy sequence in $\mathcal{M}$, and hence converges to an element $\pi \in \mathcal{M}$, and from (\ref{eq:TVconverge}), $\pi$ is independent of the initial measure $\mu$. Taking the limit as $s \rightarrow \infty$ in the above, it follows that
$$\sup_{\mu \in \mathcal{M}} \| P_t \mu - \pi \|_{TV} \leq (1 + \epsilon) e^{-\tilde{\beta} t}.$$
Hence, a Markov chain under the measure induced by $B$ is uniformly ergodic. Further, as noted above, the rate of convergence $\tilde{\beta}$ does not depend on $B$ except through $A$ and $\gamma$, and we may take $\epsilon$ arbitrarily small, with a corresponding decrease in $\tilde{\beta}$.
\end{proof}

\section{Ergodic BSDEs}

We now introduce an `Ergodic BSDE' (EBSDE) as an infinite horizon equation of the form
\begin{equation}\label{eq:ErgodicBSDE}
Y_T = Y_t - \sum_{t \leq u < T} \big( f(X_u,Z_u) - \lambda \big) + \sum_{t \leq u < T} Z_u^{\ast}M_{u+1}.
\end{equation}
A solution of this equation is a triple $(Y,Z,\lambda)$ such that (\ref{eq:ErgodicBSDE}) holds a.s. for all finite values of $t$ and $T$ such that $0 \leq t < T$, where, as usual $Y$ and $Z$ are adapted processes of appropriate dimension, and where $\lambda \in \Reals$ is a constant. Note that, unlike in our discounted BSDEs where the constant $\alpha$ is given, here $\lambda$ is to be found as part of the solution.

We will follow the method given in Section 4 of \cite{CohenHu2013}, which is itself based on the work of \cite{FuhrmanHuTessitore2009} and \cite{DebusscheHuTessitore2011}. Throughout this section we will assume that $X$ is a uniformly ergodic, time-homogeneous Markov chain on $\States$ with transition matrix $A$, and that the driver $f : \States \times \Reals^N \rightarrow \Reals$ is $\gamma$-balanced, Markovian, and independent of $y$ and $t$. Under these assumptions, we shall prove a result on the existence and uniqueness of solutions to the EBSDE (\ref{eq:ErgodicBSDE}).

\begin{mylemma}\label{XisUniErgUnderQ}
Under the assumptions stated above, let $Z,Z'$ be any two deterministic, or Markovian, $\Reals^N$-valued processes which do not depend on $t$, defined up to equivalence $\sim_M$. Then there exists a probability measure $\Q$ on $(\Omega,\F)$ such that the following conditions hold:
\begin{itemize}
\item $\Q$ is equivalent to $\Prob$ on $\F_t$ for every finite $t$,
\item $\tilde{M}_t := -\sum_{0 \leq u < t} \big(f(X_u,Z_u) - f(X_u,Z'_u)\big) + \sum_{0 \leq u < t} (Z_u - Z'_u)^{\ast}M_{u+1}$ is a martingale under $\Q$, and
\item under $\Q$, $X$ is a uniformly ergodic Markov chain, and the constants of ergodicity $R,\rho$ depend only on $A$ and $\gamma$.
\end{itemize}
\end{mylemma}

\begin{proof}
Let $\psi(\omega,t,z,z')$ be the random field associated with $f$, as given in Definition~\ref{DefnGammaBalanced}. By our assumptions on $f$ and the processes $Z,Z'$, there exists a (deterministic) matrix $\Psi^{Z,Z'}$ such that
$$\psi(\omega,t,Z_t,Z'_t) = \Psi^{Z,Z'}X_t$$
holds for all $t$. It follows from the properties of $\psi$ given in Definition~\ref{DefnGammaBalanced} that $\Psi^{Z,Z'}$ is a transition matrix, and that $\Psi^{Z,Z'} \sim_{\gamma} A$.

Let $\Q$ be the probability measure under which $X$ is a Markov chain with transition matrix $\Psi^{Z,Z'}$. Since $\Psi^{Z,Z'} \sim_{\gamma} A$, the matrices $\Psi^{Z,Z'}$ and $A$ have the same pattern of zero entries, so that the possible jumps of $X$ are the same under $\Q$ and $\Prob$. As we are in a discrete time, finite state setting, it follows immediately that $\Q$ is equivalent to $\Prob$ on $\F_t$ for every finite $t$. We also have that
\begin{align*}
\E_{\Q}\big[(Z_t - Z'_t)^{\ast}M_{t+1} \, \big| \, \F_t \big] &= (Z_t-Z'_t)^{\ast}\big(\Psi^{Z,Z'}X_t - AX_t\big) \\
&= f(X_t,Z_t) - f(X_t,Z'_t),
\end{align*}
from which it follows that $\E_{\Q}\big[\tilde{M}_{t+1} - \tilde{M}_t \, \big| \, \F_t \big] = 0$, so that $\tilde{M}$ is a martingale under $\Q$ as required. The final statement is the result of Theorem~\ref{BisUniformlyErgodic}.
\end{proof}

\begin{myremark}
As we saw in Lemma~\ref{MarkovianSolnsDisc} and Corollary~\ref{vIndepoft}, the $Z$ part of the solution of our discounted BSDE is constant (i.e. deterministic and independent of $t$), though the result above holds just as well for any Markovian processes $Z,Z'$, provided they are still independent of $t$. If they do depend on $t$ then $X$ will in general be a time-inhomogeneous Markov chain under $\Q$, in which case the notion of uniform ergodicity is no longer meaningful. Nevertheless, in an analogous context, \cite{Cohen2014} uses the continuous time version of the ergodicity estimates we obtained in Section 3 (as given in \cite{CohenHu2013}) to prove the existence of solutions of continuous time BSDEs up to unbounded stopping times.
\end{myremark}

\begin{mylemma}\label{ConvergeAlphan}
For $\alpha > 0$, let $(Y^{\alpha},Z^{\alpha})$ be the unique bounded solution (as given in Theorem~\ref{SolnDiscountedBSDE}) of the discounted BSDE
\begin{equation}\label{eq:DiscountedAlpha}
Y^{\alpha}_T = Y^{\alpha}_t - \sum_{t \leq u < T} \big( f(X_u,Z^{\alpha}_u) - \alpha Y^{\alpha}_u \big) + \sum_{t \leq u < T} \big(Z^{\alpha}_u\big)^{\ast}M_{u+1},
\end{equation}
and let $x_0 \in \States$ be an arbitrary state. By Lemma~\ref{MarkovianSolnsDisc} and Corollary~\ref{vIndepoft}, there exists a function $v^{\alpha} : \States \rightarrow \Reals$ such that $Y^{\alpha}_t = v^{\alpha}(X_t)$ and $e_k^{\ast}Z^{\alpha}_t = v^{\alpha}(e_k)$. Then there exists a bound $C' > 0$ such that
\begin{equation}\label{eq:BoundOnvalpha}
\big|v^{\alpha}(x) - v^{\alpha}(x_0)\big| \leq C', \ \ \ \ \ \alpha|v^{\alpha}(x)| \leq C'
\end{equation}
uniformly in $x$ and $\alpha$, and hence there exists a sequence $\alpha_n \searrow 0$ such that
$$\big(v^{\alpha_n}(x) - v^{\alpha_n}(x_0)\big) \rightarrow v(x) \ \ \ \ \text{and} \ \ \ \ \alpha_nv^{\alpha_n}(x) \rightarrow \lambda \ \ \ \ \text{for all} \ \ \ x \in \States,$$
for some $\lambda \in \Reals$ and some function $v : \States \rightarrow \Reals$.
\end{mylemma}

\begin{proof}
Let $C > 0$ be a bound on $|f(\cdot,0)|$. It then follows from Theorem~\ref{SolnDiscountedBSDE} that $|v^{\alpha}(\cdot)| \leq C/\alpha$. By Lemma~\ref{XisUniErgUnderQ}, there exists a measure $\Q^{\alpha}$ such that
$$-\sum_{0 \leq u < t} \big(f(X_u,Z^{\alpha}_u) - f(X_u,0)\big) + \sum_{0 \leq u < t} \big(Z^{\alpha}_u\big)^{\ast}M_{u+1}$$
is a martingale under $\Q^{\alpha}$. Moreover, $X$ is uniformly ergodic under $\Q^{\alpha}$, and the constants of ergodicity $R,\rho$ do not depend on $\alpha$. Writing $P^{\alpha}_t\delta_x$ for the law of $X_t$ under $\Q^{\alpha}$ given $X_0 = x$, we see from the proof of Theorem~\ref{BisUniformlyErgodic} that
$$\big\|P^{\alpha}_t\delta_x - P^{\alpha}_t\delta_{x'}\big\|_{TV} \leq Re^{-\rho t}$$
for all $x,x' \in \States$.
From the one-step dynamics of (\ref{eq:DiscountedAlpha}), we obtain
$$\E_{\Q^{\alpha}}[v^{\alpha}(X_{t+1}) \, | \, \F_t] = (1 + \alpha)v^{\alpha}(X_t) - f(X_t,0).$$
It follows by induction that
$$v^{\alpha}(X_0) = \E_{\Q^{\alpha}}\Bigg[(1 + \alpha)^{-T}v^{\alpha}(X_T) + \sum_{k=0}^{T-1} (1 + \alpha)^{-(k+1)}f(X_k,0) \ \Bigg| \ \F_0\Bigg]$$
for all $T \geq 1$. As $|v^{\alpha}(\cdot)| \leq C/\alpha$, letting $T \rightarrow \infty$, we deduce that
$$v^{\alpha}(x) = \lim_{T \rightarrow \infty} \E_{\Q^{\alpha}}\Bigg[\sum_{k=0}^{T-1} (1 + \alpha)^{-(k+1)}f(X_k,0) \ \Bigg| \ X_0 = x\Bigg].$$
Then, for any $x,x' \in \States$, we have
\begin{align*}
\big| v^{\alpha}(x) - v^{\alpha}(x') \big| &= \Bigg| \lim_{T \rightarrow \infty} \E_{\Q^{\alpha}}\Bigg[\sum_{k=0}^{T-1} (1 + \alpha)^{-(k+1)}f(X_k,0) \ \Bigg| \ X_0 = x\Bigg] \\
&\ \ \ \ \ \ - \lim_{T \rightarrow \infty} \E_{\Q^{\alpha}}\Bigg[\sum_{k=0}^{T-1} (1 + \alpha)^{-(k+1)}f(X_k,0) \ \Bigg| \ X_0 = x'\Bigg] \Bigg| \\
&= \lim_{T \rightarrow \infty} \Bigg| \sum_{k=0}^{T-1} (1 + \alpha)^{-(k+1)} \int_{\States} f(X_k,0) \big( d(P^{\alpha}_k\delta_x) - d(P^{\alpha}_k\delta_{x'}) \big) \Bigg| \\
&\leq \lim_{T \rightarrow \infty} \, 2C \sum_{k=0}^{T-1} (1 + \alpha)^{-(k+1)} \big\|P^{\alpha}_k\delta_x - P^{\alpha}_k\delta_{x'}\big\|_{TV} \\
&\leq \lim_{T \rightarrow \infty} \, 2CR \sum_{k=0}^{T-1} (1 + \alpha)^{-(k+1)}e^{-\rho k} = \frac{2CR}{1 + \alpha - e^{-\rho}}.
\end{align*}
We therefore have the bound
$$\big| v^{\alpha}(x) - v^{\alpha}(x') \big| \leq \frac{2CR}{1 - e^{-\rho}},$$
which holds for all $x,x' \in \States$ and all $\alpha > 0$. Since $\alpha|v^{\alpha}(\cdot)| \leq C$ (and we may certainly assume that $R > 1$), we see that (\ref{eq:BoundOnvalpha}) is satisfied with $C' = 2CR(1 - e^{-\rho})^{-1}$.

By the Bolzano--Weierstrass Theorem, there exists a sequence $\alpha_n \searrow 0$ such that $\alpha_nv^{\alpha_n}(x_0) \rightarrow \lambda$ and
$$\big(v^{\alpha_n}(x) - v^{\alpha_n}(x_0)\big) \rightarrow v(x)$$
for all $x \in \States$, for some $\lambda \in \Reals$ and some function $v : \States \rightarrow \Reals$. Note that $|\lambda|$ and $|v(\, \cdot \, )|$ are both bounded by $C'$. Finally, we notice that, for any $x \in \States$,
$$\alpha_nv^{\alpha_n}(x) = \alpha_nv^{\alpha_n}(x_0) + \alpha_n\big(v^{\alpha_n}(x) - v^{\alpha_n}(x_0)\big) \rightarrow \lambda,$$
so the convergence of this sequence to $\lambda$ holds for all $x$.
\end{proof}

We are now in a position to prove existence of solutions to our Ergodic BSDE.

\begin{mythm}\label{SolnErgodicBSDE}
Let $v$ and $\lambda$ be as constructed in Lemma~\ref{ConvergeAlphan}. The triple $(Y,Z,\lambda)$, where
$$Y_t := v(X_t), \ \ \ \ \ \ e_k^{\ast}Z_t := v(e_k),$$
is the unique bounded, stationary (i.e. does not depend on $t$), Markovian solution, with $v(x_0) = 0$, to the Ergodic BSDE
\begin{equation}\label{eq:ErgodicBSDEthm}
Y_T = Y_t - \sum_{t \leq u < T} \big( f(X_u,Z_u) - \lambda \big) + \sum_{t \leq u < T} Z_u^{\ast}M_{u+1}.
\end{equation}
Any other bounded solution $(Y',Z',\lambda')$ satisfies $\lambda = \lambda'$, and any other bounded, stationary, Markovian solution $(Y',Z',\lambda')$ satisfies $Y_t = Y'_t + c$ for some $c \in \Reals$, and $Z \sim_M Z'$.
\end{mythm}

\begin{proof}
Let $\{\alpha_n\}_{n \geq 1}$ be the sequence constructed in Lemma~\ref{ConvergeAlphan}. We have that $Y^{\alpha_n}_t = v^{\alpha_n}(X_t)$ and $e_k^{\ast}Z^{\alpha_n}_t = v^{\alpha_n}(e_k)$ solve the discounted BSDE
$$Y^{\alpha_n}_T = Y^{\alpha_n}_t - \sum_{t \leq u < T} \big( f(X_u,Z^{\alpha_n}_u) - \alpha_n Y^{\alpha_n}_u \big) + \sum_{t \leq u < T} \big(Z^{\alpha_n}_u\big)^{\ast}M_{u+1}.$$
However, since $\|\mathbf{1}\|_{M_{t+1}} = 0$, and $f$ does not distinguish between values of $Z_u$ up to equivalence $\sim_{M_{u+1}}$, we can equally write $Z^{\alpha_n}_u - v^{\alpha_n}(x_0)\mathbf{1}$ in the place of $Z^{\alpha_n}_u$ in the above. Note that $e_k^{\ast}Z^{\alpha_n}_u - v^{\alpha_n}(x_0) \rightarrow e_k^{\ast}Z_u$ as $n \rightarrow \infty$ for each $e_k \in \States$, and that, by the bound established in Lemma~\ref{ConvergeAlphan}, $|e_k^{\ast}Z^{\alpha_n}_u - v^{\alpha_n}(x_0)|$ is uniformly bounded. Since $f$ is Lipschitz in $z$, we deduce that
$$f(X_u,Z^{\alpha_n}_u - v^{\alpha_n}(x_0)\mathbf{1}) \rightarrow f(X_u,Z_u) \ \ \ \text{as} \ \ \ n \rightarrow \infty \ \ \ \text{a.s.}$$
It follows that
\begin{align*}
Y_T &= \lim_{n \rightarrow \infty} \big(v^{\alpha_n}(X_T) - v^{\alpha_n}(x_0)\big) \\
&= \lim_{n \rightarrow \infty} \big(v^{\alpha_n}(X_t) - v^{\alpha_n}(x_0)\big) - \lim_{n \rightarrow \infty} \sum_{t \leq u < T} \big( f(X_u,Z^{\alpha_n}_u) - \alpha_n v^{\alpha_n}(X_u) \big) \\
&\ \ \ \ \ \ + \lim_{n \rightarrow \infty} \sum_{t \leq u < T} \big(Z^{\alpha_n}_u\big)^{\ast}M_{u+1} \\
&= v(X_t) - \lim_{n \rightarrow \infty} \sum_{t \leq u < T} \big( f(X_u,Z^{\alpha_n}_u - v^{\alpha_n}(x_0)\mathbf{1}) - \alpha_n v^{\alpha_n}(X_u) \big) \\
&\ \ \ \ \ \ + \lim_{n \rightarrow \infty} \sum_{t \leq u < T} \big(Z^{\alpha_n}_u - v^{\alpha_n}(x_0)\mathbf{1}\big)^{\ast}M_{u+1} \\
&= Y_t - \sum_{t \leq u < T} \big( f(X_u,Z_u) - \lambda \big) + \sum_{t \leq u < T} Z_u^{\ast}M_{u+1},
\end{align*}
and we see that $(Y,Z,\lambda)$ is indeed a solution of the EBSDE (\ref{eq:ErgodicBSDEthm}).

Suppose that $(Y',Z',\lambda')$ is another bounded solution. Let $\tilde{Y} = Y - Y'$, $\tilde{Z} = Z - Z'$ and $\tilde{\lambda} = \lambda - \lambda'$. Then
\begin{equation}\label{eq:TildeDynamics}
\tilde{Y}_T = \tilde{Y}_0 - \sum_{0 \leq u < T} \big( f(X_u,Z_u) - f(X_u,Z'_u) - \tilde{\lambda} \big) + \sum_{0 \leq u < T} \tilde{Z}_u^{\ast}M_{u+1}.
\end{equation}
By Proposition~\ref{ChangeMeasure}, there exists a measure $\Q$ such that
$$-\sum_{0 \leq u < t} \big(f(X_u,Z_u) - f(X_u,Z'_u)\big) + \sum_{0 \leq u < t} \tilde{Z}_u^{\ast}M_{u+1}$$
is a martingale under $\Q$. Taking an $\E_{\Q}$ expectation in (\ref{eq:TildeDynamics}), we obtain
$$\tilde{\lambda} = T^{-1}\E_{\Q}\big[\tilde{Y}_T - \tilde{Y}_0\big].$$
Since $\tilde{Y}$ is uniformly bounded, taking the limit $T \rightarrow \infty$ gives $\tilde{\lambda} = 0$, so that $\lambda = \lambda'$. Substituting back into (\ref{eq:TildeDynamics}) and taking an $\E_{\Q}[\, \cdot \, | \, \F_0]$ expectation gives
\begin{equation}\label{eq:YtildeCondExp}
\E_{\Q}\big[\tilde{Y}_T \, \big| \, \F_0\big] = \tilde{Y}_0.
\end{equation}
Suppose further that $Y'$ and $Z'$ are Markovian and do not depend on $t$, so that in particular there exists a function $v' : \States \rightarrow \Reals$ such that $Y'_t = v'(X_t)$. Then the measure $\Q$ may be taken to be the measure given by Lemma~\ref{XisUniErgUnderQ}, so that $X$ is still a uniformly ergodic Markov chain under $\Q$. Writing $\tilde{\pi}$ for the ergodic measure of $X$ under $\Q$, it follows from (\ref{eq:YtildeCondExp}) that, for any $x \in \States$,
$$v(x) - v'(x) = \lim_{T \rightarrow \infty} \E_{\Q}\big[\tilde{Y}_T \, \big| \, X_0 = x\big] = \int_{\States} \big(v(y) - v'(y)\big) d\tilde{\pi}(y).$$
Since the right hand side is independent of $x$, we see that $v(x) = v'(x) + c$ for all $x$, and hence that $Y_t = Y'_t + c$, for some $c \in \Reals$. In particular, if $v'(x_0) = 0$, then $c = 0$, and hence $Y = Y'$ up to indistinguishability.

With $\tilde{\lambda} = 0$ and $\tilde{Y}_t = c = \tilde{Y}_{t+1}$, we deduce from the one-step dynamics of $\tilde{Y}$ that $\tilde{Z}_t^{\ast}M_{t+1} = 0$, and hence that $Z \sim_M Z'$.
\end{proof}

\begin{mycor}
The sequences $\{\alpha_n v^{\alpha_n}(x)\}_{n \geq 1}$ and $\{v^{\alpha_n}(x) - v^{\alpha_n}(x_0)\}_{n \geq 1}$ constructed in Lemma~\ref{ConvergeAlphan} converge to $\lambda$ and $v(x)$ respectively for any choice of sequence $\alpha_n \searrow 0$.
\end{mycor}

\begin{proof}
Suppose this were not the case. Then we could construct two sequences with distinct limits, both of which would yield bounded, stationary, Markovian solutions to the EBSDE, contradicting the uniqueness established in Theorem~\ref{SolnErgodicBSDE}.
\end{proof}

\begin{mycor}
Let $\pi$ denote the ergodic measure of $X$ under $\Prob$. Writing $\mathbf{v}$ for the vector in $\Reals^N$ with entries $e_k^{\ast}\mathbf{v} = v(e_k)$, the value $\lambda$ in the EBSDE solution $(Y,Z,\lambda) = (v(X_t),\mathbf{v},\lambda)$ is given by
$$\lambda = \int_{\States} f(x,\mathbf{v}) d\pi(x) = \sum_{x \in \States} f(x,\mathbf{v}) \pi(\{x\}).$$
Furthermore, there exists a probability measure $\pi_{\mathbf{v}}$ on $\States$ such that
$$\lambda = \int_{\States} f(x,0) d\pi_{\mathbf{v}}(x).$$
\end{mycor}

\begin{proof}
The invariance of the ergodic measure $\pi$ implies that for any fixed time $t \geq 0$ and any function $g : \States \rightarrow \Reals$,
\begin{equation}\label{eq:InvarianceProperty}
\int_{\States} \E\big[g(X_t) \, \big| \, X_0 = x\big] d\pi(x) = \int_{\States} g(x) d\pi(x).
\end{equation}
For any $T > 0$, we have
\begin{equation}\label{EBSDEsoln0T}
v(X_0) = v(X_T) + \sum_{0 \leq u < T} \big(f(X_u,\mathbf{v}) - \lambda\big) - \sum_{0 \leq u < T} \mathbf{v}^{\ast}M_{u+1}.
\end{equation}
Then, by the invariance property (\ref{eq:InvarianceProperty}),
\begin{align*}
\int_{\States} v(x) d\pi(x) &= \int_{\States} \E\Bigg[v(X_T) + \sum_{0 \leq u < T} \big(f(X_u,\mathbf{v}) - \lambda\big) \, \Bigg| \, X_0 = x\Bigg] d\pi(x) \\
&= \int_{\States} v(x) d\pi(x) + T\int_{\States} f(x,\mathbf{v}) d\pi(x) - T\lambda,
\end{align*}
and rearranging gives the first result.

By Lemma~\ref{XisUniErgUnderQ}, there exists a probability measure $\Q$ such that, under $\Q$,
$$-\sum_{0 \leq u < t} \big(f(X_u,\mathbf{v}) - f(X_u,0)\big) + \sum_{0 \leq u < t} \mathbf{v}^{\ast}M_{u+1}$$
is a martingale and $X$ is still a uniformly ergodic Markov chain. Let $\pi_{\mathbf{v}}$ denote the ergodic measure of $X$ under $\Q$. From (\ref{EBSDEsoln0T}), and the invariance property (\ref{eq:InvarianceProperty}), we have
\begin{align*}
\int_{\States} v(x) d\pi_{\mathbf{v}}(x) &= \int_{\States} \E_{\Q}\Bigg[v(X_T) + \sum_{0 \leq u < T} \big(f(X_u,0) - \lambda\big) \, \Bigg| \, X_0 = x\Bigg] d\pi_{\mathbf{v}}(x) \\
&= \int_{\States} v(x) d\pi_{\mathbf{v}}(x) + T\int_{\States} f(x,0) d\pi_{\mathbf{v}}(x) - T\lambda,
\end{align*}
and rearranging gives the second result.
\end{proof}

\begin{myremark}
Under the additional assumption that $X$ has no transient states, it follows from the one-step dynamics of (\ref{eq:ErgodicBSDEthm}), that the EBSDE solution $(v(X_t),\mathbf{v},\lambda)$ is also a solution of the vector equation
\begin{equation}\label{OneStepErgvSoln}
\mathbf{v} - \mathbf{f}(\mathbf{v}) + \lambda \mathbf{1} - A^{\ast}\mathbf{v} = 0,
\end{equation}
where $\mathbf{f}(\mathbf{v})$ denotes the vector with entries $e_k^{\ast}\mathbf{f}(\mathbf{v}) = f(e_k,\mathbf{v})$. Further, this solution is unique up to equality in $\lambda$ and a constant shift in $\mathbf{v}$.
\end{myremark}

We now provide a comparison theorem for Ergodic BSDEs, concerning the $\lambda$ part of the solution.

\begin{mythm}
Let $f$ and $f'$ be two $\gamma$-balanced Markovian drivers, and let $(Y,Z,\lambda)$ and $(Y',Z',\lambda')$ be any corresponding bounded EBSDE solutions. If $f(x,z) \geq f'(x,z)$ for all $x \in \States$ and $z \in \Reals^N$, then $\lambda \geq \lambda'$.
\end{mythm}

\begin{proof}
Write $(\bar{Y},\bar{Z},\lambda) = (v(X_t),\mathbf{v},\lambda)$ and $(\bar{Y}',\bar{Z}',\lambda') = (v'(X_t),\mathbf{v}',\lambda')$ for the corresponding bounded, stationary, Markovian solutions given in Theorem~\ref{SolnErgodicBSDE}. By Lemma~\ref{XisUniErgUnderQ}, there exists a probability measure $\Q$ such that, under $\Q$,
$$-\sum_{0 \leq u < t} \big(f'(X_u,\mathbf{v}) - f'(X_u,\mathbf{v}')\big) + \sum_{0 \leq u < t} (\mathbf{v} - \mathbf{v}')^{\ast}M_{u+1}$$
is a martingale and $X$ is still a uniformly ergodic Markov chain. Write $\hat{\pi}$ for the ergodic measure of $X$ under $\Q$. Taking the difference of the EBSDEs, and applying the invariance property (\ref{eq:InvarianceProperty}), we have
\begin{align*}
\int_{\States}& \big(v(x) - v'(x)\big) d\hat{\pi}(x) \\
&= \int_{\States} \E_{\Q}\big[v(X_T) - v'(X_T) \, \big| \, X_0 = x\big] d\hat{\pi}(x) \\
&\ \ \ \ \ \ + \int_{\States} \E_{\Q}\Bigg[\sum_{0 \leq u < T} \big(f(X_u,\mathbf{v}) - f'(X_u,\mathbf{v}) - (\lambda - \lambda')\big) \, \Bigg| \, X_0 = x\Bigg] d\hat{\pi}(x) \\
&= \int_{\States} \big(v(x) - v'(x)\big) d\hat{\pi}(x) + T\int_{\States} \big(f(x,\mathbf{v}) - f'(x,\mathbf{v})\big) d\hat{\pi}(x) - T(\lambda - \lambda'),
\end{align*}
and hence, by rearrangement,
$$\lambda - \lambda' = \int_{\States} \big(f(x,\mathbf{v}) - f'(x,\mathbf{v})\big) d\hat{\pi}(x) \geq 0.$$
\end{proof}

\section{Optimal Ergodic Control}

As an application of the theory we have just developed for discrete time Ergodic BSDEs, we shall now present a treatment of an ergodic control problem. We will follow the method given in Section 5.2 of \cite{CohenHu2013}, which is itself based on the work of \cite{FuhrmanHuTessitore2009} and \cite{DebusscheHuTessitore2011}.

Suppose that
\begin{itemize}
\item Under $\Prob$, $X$ is a uniformly ergodic Markov chain on $\States$ with transition matrix $A$,
\item $\mathcal{U}$ is a space of `controls', which we assume to be a topological space, equal to a countable union of compact metrizable subsets of itself,
\item $B^{(\cdot)}$ is a continuous function which maps each element $u \in \mathcal{U}$ to a transition matrix $B^u$ such that $B^u \sim_{\gamma} A$, for some fixed $\gamma \in (0,1)$,
\item $L : \States \times \mathcal{U} \rightarrow \Reals$ is a cost function, which we assume to be bounded, measurable in $x$ and continuous in $u$.
\end{itemize}
Define the ergodic cost
$$J(x,U) = \limsup_{T \rightarrow \infty} \, \frac{1}{T} \, \E^{U}_x\Bigg[\sum_{0 \leq s < T} L(X_s,U_s)\Bigg],$$
where
\begin{itemize}
\item $U$ is a $\mathcal{U}$-valued adapted process, which we shall also refer to as a `control',
\item $\E^U_x$ is the expectation under which $X_0 = x$, and for any given pair $(\omega,t)$, $X$ jumps according to transition matrix $B^{U_t(\omega)}$.
\end{itemize}

We wish to minimize $J(x,U)$ over all controls $U$. Our approach is to work with the probability measure under which $X$ jumps at time $t$ according to the transition matrix $B^{U_t}$. As we have made no Markov or time-homogeneity assumptions on $U$, $X$ will in general not be a Markov chain under our new measure.

We define the Hamiltonian
\begin{equation}\label{eq:DefnHamiltonian}
f(x,z) = \inf_{u \in \mathcal{U}} \big\{L(x,u) + z^{\ast}(B^u - A)x\big\}.
\end{equation}
Note that $f$ is finite valued and that $f(\, \cdot \,,0)$ is bounded. By Lemma~\ref{infGammaBalLemma}, $f(X_t,z)$ is a $\gamma$-balanced Markovian driver. Hence, by Theorem~\ref{SolnErgodicBSDE}, the EBSDE with driver $f(X_t,z)$ admits a bounded, stationary, Markovian solution $(\bar{Y},\bar{Z},\bar{\lambda}) = (v(X_t),\mathbf{v},\bar{\lambda})$, where as usual $\mathbf{v}$ is the vector in $\Reals^N$ with components $v(e_k)$.

If the infimum in (\ref{eq:DefnHamiltonian}) is attained then, by Filippov's implicit function theorem (see either McShane and Warfield \cite{McShaneWarfield1967} or Bene\v{s} \cite{Benes1970}), there exists a measurable function $\kappa : \States \times \Reals^N \rightarrow \mathcal{U}$ such that
\begin{equation}\label{eq:measkappa}
f(x,z) = L(x,\kappa(x,z)) + z^{\ast}(B^{\kappa(x,z)} - A)x.
\end{equation}

If the infimum in (\ref{eq:DefnHamiltonian}) is not attained, then applying Theorem 21.3.4 from \cite{CohenElliott2015} to the function $G(x,z,u) = L(x,u) + z^{\ast}(B^u - A)x$, we deduce, for any $\epsilon > 0$, the existence of a measurable function $\tau^{\epsilon} : \States \times \Reals^N \rightarrow \mathcal{U}$ such that
\begin{equation}\label{eq:EpsilonOptimal}
L(x,\tau^{\epsilon}(x,z)) + z^{\ast}(B^{\tau^{\epsilon}(x,z)} - A)x < f(x,z) + \epsilon.
\end{equation}

\begin{mythm}\label{ErgodicControlThm}
In the setting described above, let $(Y,Z,\lambda)$ be any (possibly non-Markovian) bounded solution to the EBSDE (\ref{eq:ErgodicBSDE}) with driver $f$. Then the following hold:
\begin{itemize}
\item For any control $U$, we have $J(x,U) \geq \lambda = \bar{\lambda}$, with equality if
\begin{equation}\label{eq:infattained}
f(X_s,Z_s) = L(X_s,U_s) + Z_s^{\ast}\big(B^{U_s} - A\big)X_s \ \ \text{for all} \ \ s \geq 0.
\end{equation}
\item If the infimum in (\ref{eq:DefnHamiltonian}) is attained, then the control $\bar{U}_t = \kappa(X_t,Z_t)$ satisfies $J(x,\bar{U}) = \bar{\lambda}$.
\item Even if the infimum in (\ref{eq:DefnHamiltonian}) is not attained, there still exists a control $\hat{U}$ such that $J(x,\hat{U}) = \bar{\lambda}$.
\end{itemize}
\end{mythm}

\begin{proof}
The fact that $\lambda = \bar{\lambda}$ follows immediately from Theorem~\ref{SolnErgodicBSDE}.

Note that
$$-\sum_{0 \leq s < t} Z_s^{\ast}\big(B^{U_s} - A\big)X_s + \sum_{0 \leq s < t} Z_s^{\ast}M_{s+1}$$
is a martingale under $\E^U_x$. Since $(Y,Z,\bar{\lambda})$ is a solution of the EBSDE, we have
$$Y_T = Y_0 - \sum_{0 \leq s < T}\big(f(X_s,Z_s) - \bar{\lambda}\big) + \sum_{0 \leq s < T}Z_s^{\ast}M_{s+1}$$
for any $T>0$. Then
\begin{align}
\bar{\lambda} &= \frac{1}{T} \, \E^U_x\Bigg[Y_T - Y_0 + \sum_{0 \leq s < T} L(X_s,U_s)\Bigg] \nonumber \\
&\ \ \ \ \ \ + \frac{1}{T} \, \E^U_x\Bigg[\sum_{0 \leq s < T} \big(f(X_s,Z_s) - L(X_s,U_s) - Z_s^{\ast}\big(B^{U_s} - A\big)X_s\big)\Bigg]. \label{eq:lambdabar}
\end{align}
Since $f$ is the infimum over all controls, we have
$$\bar{\lambda} \leq \frac{1}{T} \, \E^U_x\Bigg[Y_T - Y_0 + \sum_{0 \leq s < T} L(X_s,U_s)\Bigg].$$
As $Y_T - Y_0$ is uniformly bounded, it follows that
$$\bar{\lambda} \leq \limsup_{T \rightarrow \infty} \, \frac{1}{T} \, \E^U_x\Bigg[\sum_{0 \leq s < T} L(X_s,U_s)\Bigg] = J(x,U).$$

If (\ref{eq:infattained}) holds, then equality holds throughout the above, and we see that $\bar{\lambda} = J(x,U)$.

If the infimum in (\ref{eq:DefnHamiltonian}) is attained, then we know that (\ref{eq:measkappa}) holds for some measurable function $\kappa : \States \times \Reals^N \rightarrow \mathcal{U}$. Setting $\bar{U}_s = \kappa(X_s,Z_s)$, we have that (\ref{eq:infattained}) holds, so that $\bar{\lambda} = J(x,\bar{U})$.

Even if the infimum in (\ref{eq:DefnHamiltonian}) is not attained, by (\ref{eq:EpsilonOptimal}), there exists a measurable function $\tau : \mathbb{N} \times \States \times \Reals^N \rightarrow \mathcal{U}$ such that
$$L(x,\tau(s,x,z)) + z^{\ast}(B^{\tau(s,x,z)} - A)x < f(x,z) + 2^{-s}.$$
Setting $\hat{U}_s = \tau(s,X_s,Z_s)$, (\ref{eq:lambdabar}) reduces to
$$\bar{\lambda} \geq \frac{1}{T} \, \E^{\hat{U}}_x\Bigg[Y_T - Y_0 + \sum_{0 \leq s < T} L(X_s,\hat{U}_s)\Bigg] - \frac{1}{T} \sum_{0 \leq s < T} 2^{-s},$$
and it follows that $\bar{\lambda} = J(x,\hat{U})$.
\end{proof}

In particular, the controls $\bar{U}_s = \kappa(X_s,\mathbf{v})$ (assuming the function $\kappa$ exists) and $\hat{U}_s = \tau(s,X_s,\mathbf{v})$ are optimal feedback controls.

\begin{myremark}
 We note that $\hat{U}_s$ is time-dependent, but that for any $\epsilon > 0$, a time-homogenous feedback control satisfying $J(x,\hat{U}) \leq \bar{\lambda} + \epsilon$ can also be attained by setting $\hat{U}_s = \tau^{\epsilon}(X_s,\mathbf{v})$, where $\tau^{\epsilon}$ satisfies (\ref{eq:EpsilonOptimal}).
\end{myremark}

It is worth noticing that the $\lambda$ we have obtained is optimal in the class of all strategies, not only among those of feedback type. If we had approached our problem purely through the equations obtained from (\ref{OneStepErgvSoln}) and (\ref{eq:measkappa}), this optimality would require separate analysis.

As noted above, (\ref{eq:infattained}) is a sufficient condition to guarantee that $J(x,U) = \bar{\lambda}$. In \cite{FuhrmanHuTessitore2009} and \cite{CohenHu2013}, among others, in analogous contexts, it is stated that this is also a necessary condition. However, this is incorrect, as modifying the value of an optimal control over any finite time horizon does not affect the associated ergodic cost. Indeed, in our setting we have shown that we can always construct an optimal control such that (\ref{eq:infattained}) does not necessarily hold at any time $s$.

\begin{myremark}
We have assumed that the Markov chain $X$ and the cost function $L$ are time-homogeneous. However, by considering the cyclic classes of $X$, it is not hard to extend the result of Theorem~\ref{ErgodicControlThm} to the case when both $X$ and $L$ are periodic in time.
\end{myremark}

\section{Conclusion}

We have shown that, when our underlying process is a uniformly ergodic Markov chain, Ergodic BSDEs with $\gamma$-balanced, Markovian drivers admit bounded Markovian solutions. We have also shown how these equations arise in the context of control problems when considering an ergodic cost functional, so that it is the long term asymptotic behaviour of the process that is important.

This has involved extending the existing theory of discrete time BSDEs to an infinite time horizon setting, in order to prove the existence of unique bounded solutions to a class of discounted BSDEs. We then showed how, by taking a suitable limit in these equations, we can obtain solutions to corresponding Ergodic BSDEs. A notable step in the proof involved finding ergodicity estimates for a discrete time Markov chain which hold uniformly for a suitable class of transition matrices, a result which is interesting in its own right.

When considering Ergodic BSDEs, we made Markovian and time-homogeneity assumptions on the underlying process $X$ and on our driver function $f$. However, we note that a significant amount of the preceding theory, including our construction of solutions to discounted BSDEs, holds much more generally, without any Markovian assumptions.

Throughout, we have considered discrete time (E)BSDEs as entities in their own right, rather than as approximations to the continuous time case. The precise connections and the extent to which our results approximate the corresponding continuous time theory remain to be explored.

One extension of our theory would be to the case where the underlying process is defined on a countably infinite state space, though it is expected that the majority of our analysis will remain essentially the same.

\end{document}